\documentclass[11pt,twoside]{article}

\usepackage{amsmath,amsthm,amssymb}
\usepackage{color} 
\usepackage{mathtools,mathptmx}
\usepackage{indentfirst}
\usepackage{mathptmx}

\usepackage{amsbsy}  
\usepackage{amsfonts}
\usepackage{graphicx}
\usepackage{tkz-euclide}

\usepackage{array}
\usepackage{geometry}
\usepackage[bookmarks=true,colorlinks=true, pdfstartview=FitV, linkcolor=black, citecolor=blue, urlcolor=black]{hyperref}

\usepackage{cases}
\mathtoolsset{showonlyrefs}

\newtheorem{theorem}{Theorem}[section]
\newtheorem{proposition}[theorem]{Proposition}

\newtheorem{lemma}[theorem]{Lemma}

\numberwithin{equation}{section}
\numberwithin{theorem}{section}

\usepackage{fancyhdr}
\newcommand{\dis}{\displaystyle}
\newcommand{\R}{\mathbb{R}}
\newcommand{\C}{\mathbb{C}}
\newcommand{\Z}{\mathbb{Z}}
\newcommand{\N}{\mathbb{N}}

\newcommand{\parameter}[1]{{\left[#1\right]}}
\newcommand{\roundparameter}[1]{{\left(#1\right)}}

\newcommand{\kk}{{\parameter{k}}}

\newcommand{\dd}{\mathrm{d}}
\newcommand{\dx}[1][x]{\mathrm{d}#1}
\newcommand{\DiffOp}[1][x]{\frac{\dd}{\dx[#1]}}
\newcommand{\DiffOpHigherOrder}[2][x]{\frac{\dd^{#2}}{\dx[#1]^{#2}}}
\newcommand{\DiffOpFunction}[2][x]{\frac{\dd #2}{\dx[#1]}}

\newcommand{\pochhammer}[2][n]{\left(#2\right)_{#1}}

\newcommand{\Hypergeometric}[5][x]
{{}_{#2} F_{#3} \left(\begin{matrix} #4 \\ #5 \end{matrix} \, ; \, #1\right)}
\newcommand{\twoFtwo}[3][x]{\Hypergeometric[#1]{2}{2}{#2}{#3}}
\newcommand{\twoFone}[3][x]{\Hypergeometric[#1]{2}{1}{#2}{#3}}
\newcommand{\threeFtwo}[3][x]{\Hypergeometric[#1]{3}{2}{#2}{#3}}

\newcommand{\HypergeometricOneLine}[5][x]{{}_{#2} F_{#3} \left(#4;#5;#1\right)}

\newcommand{\twoFoneOneLine}[3][x]{\HypergeometricOneLine[#1]{2}{1}{#2}{#3}}

\newcommand{\W}{\mathcal{W}}
\newcommand{\Wvec}{\overline{\W}}
\newcommand{\V}{\mathcal{V}}

\newcommand{\KummerU}[3][x]{\textbf{U}\left(#2,#3;#1\right)}

\newcommand\matrixtwobytwo[4]{\begin{bmatrix} \dis #1 & \dis #2 \vspace*{0,1 cm} \\ 
		\dis #3 & \dis #4 \end{bmatrix}}

\newcommand\twovector[2]
{\begin{bmatrix} \dis #1 \vspace*{0,1 cm} \\ \dis #2 \end{bmatrix}}

\newcommand{\floor}[1]{\left\lfloor #1 \right\rfloor}

\newcommand{\polyseq}[1][P_n(x)]{\dis\left\{#1\right\}_{n\in\N}}

\newcommand{\KK}{\huge \textbf{K} \normalsize}
\newcommand{\contfrac}[3][j=1]{\mathop{\KK}\limits_{#1}^{\infty}\left(\frac{#2}{#3}\right)}

\newcommand{\n}{\vec{n}}

\newcommand{\e}{\mathrm{e}}

\begin{document}
	
\title{\vspace*{-2cm}Multiple orthogonal polynomials with respect to Gauss' hypergeometric function}
\author{H\'elder Lima\footnote{Address: School of Mathematics, Statistics and Actuarial Sciences, University of Kent, Sibson Building, Parkwood Road, Canterbury, CT2 7FS,  UK \newline 
Email addresses: (H\'elder Lima) has27@kent.ac.uk and (Ana Loureiro) a.loureiro@kent.ac.uk }\  \ and Ana Loureiro$^{*}$}

\date{\today}
\maketitle\vspace*{-1cm}	

\abstract{
\noindent A new set of multiple orthogonal polynomials of both type I and type II with respect to two  weight functions involving Gauss' hypergeometric function on the interval $(0,1)$ is studied. 
This type of polynomials have direct applications in the investigation of singular values of products of Ginibre matrices, in the analysis of rational solutions to Painlev\'e equations and are connected with branched continued fractions and total positivity problems in combinatorics.  
The pair of orthogonality measures is shown to be a Nikishin system and to satisfy a matrix Pearson-type differential equation.
The focus is on the polynomials whose indexes lie on the step line, for which it is shown that a differentiation on the variable gives a shift on the parameters, therefore satisfying the Hahn's property.
We obtain a Rodrigues-type formula for the type I, while a more detailed characterisation is given for the type II polynomials (aka $2$-orthogonal polynomials) which includes: an explicit expression as a terminating hypergeometric series, a third-order differential equation and a third-order recurrence relation.
The asymptotic behaviour of their recurrence coefficients mimics those of Jacobi-Pi\~neiro polynomials, based on which, their zero asymptotic distribution and a Mehler-Heine asymptotic formula near the origin are given.
Particular choices on the parameters degenerate in some known systems such as special cases of the Jacobi-Pi\~neiro polynomials, Jacobi-type $2$-orthogonal polynomials and components of the cubic decomposition of threefold symmetric Hahn-classical polynomials. 
Equally considered are confluence relations to other known polynomial sets, such as multiple orthogonal polynomials with respect to Tricomi functions.\\}



\color{black}

\section{Introduction and motivation}

The main aim of this paper is to investigate the multiple orthogonal polynomials with respect to two absolutely continuous measures supported on the interval $(0,1)$ and admitting an integral representation via weight functions $\W(x;a,b;c,d)$ and $\W(x;a,b+1;c+1,d)$, where
\begin{align}
\label{hypergeometric weight definition}
\W(x;a,b;c,d)=\frac{\Gamma(c)\Gamma(d)}{\Gamma(a)\Gamma(b)\Gamma(\delta)}\,x^{a-1}(1-x)^{\delta-1}\,\twoFone[1-x]{c-b,d-b}{\delta},
\end{align}
with  
\begin{equation}\label{parameters}
a,b,c,d\in\R^+\ \text{ such that } \ \min\{c,d\}>\max\{a,b\} \ \text{ and } \delta=c+d-a-b>0.
\end{equation}

The weight functions involve Gauss' hypergeometric function, which is defined, for parameters $\alpha,\beta\in\C$ and $\gamma\in\C\backslash\{-n:n\in\N\}$, by
\begin{align}
\label{2F1 definition}
\twoFoneOneLine[z]{\alpha,\beta}{\gamma}=\sum\limits_{n=0}^{\infty}\frac{\pochhammer{\alpha}\pochhammer{\beta}}{\pochhammer{\gamma}}\frac{z^n}{n!}.
\end{align}
where $\pochhammer{z}$ denotes the Pochhammer symbol defined by 
\begin{align*}
\pochhammer[0]{z}=1 
\quad \text{and} \quad 
\pochhammer[n]{z}:=z(z+1)\cdots(z+n-1),
\quad n\in\Z^+.
\end{align*} 

The detailed knowledge of multiple orthogonal polynomials with respect to (generalised) hypergeometric functions has applications in random matrix theory, combinatorics, description of rational solutions to nonlinear differential difference equations, such as Painlev\'e equations, number theory, among other fields. 
For instance, the analysis of singular values of products of Ginibre matrices in \cite{KuijlaarsZhang14,KuijlaarsStiv14} uses multiple orthogonal polynomials associated with weight functions expressed in terms of Meijer G-functions, a class of weights to which the weight \eqref{hypergeometric weight definition} belongs. 
Besides, these polynomials are linked with the branched continued fractions introduced in \cite{AlanSokalEtAlBranchedContinuedFractions} as the generating functions of $m$-Dyck paths, for the purpose of solving total positivity problems involving combinatorially interesting sequences of polynomials.
This connection, which leads to new results on both fields involved, will be further explored in forthcoming work.

The hypergeometric function defined by \eqref{2F1 definition} converges absolutely for {$|z|< 1$}, 
and it is a solution of the hypergeometric differential equation
\begin{align}
\label{hypergeometric differential equation}
z(1-z)F''(z)+(\gamma-(\alpha+\beta+1)z)F'(z)-\alpha\beta F(z)=0.
\end{align}
Recall the identity  
$
	\twoFoneOneLine[1]{\alpha,\beta}{\gamma} = \frac{\Gamma(\gamma) \Gamma(\gamma-\alpha-\beta)}{\Gamma(\gamma-\beta)\Gamma(\gamma-\alpha)}
$, which is valid  for $(\gamma-\alpha-\beta)>0$. When  $(\gamma-\alpha-\beta)<0$, we have 
 $\lim\limits_{x\to1^{-}} (1-x)^{-(\gamma-\alpha-\beta)}{ \twoFoneOneLine[x]{\alpha,\beta}{\gamma}}
 = \frac{\Gamma(\gamma)\Gamma(\alpha+ \beta-\gamma)}{\Gamma(\alpha)\Gamma(\beta)}$. 
This yields 
 $\lim\limits_{x\to 0^+} \W(x;a,b;c,d) =0$.

Observe that $\W(x;a,b;c,d)=\W(x;a,b;d,c)$, which is a straightforward consequence of \eqref{hypergeometric weight definition} and \eqref{2F1 definition}. In addition, the symmetry  $\dis\W(x;a,b;c,d)=\W(x;b,a;c,d)$ also holds, because using \cite[Eq.~15.8.1]{DLMF} we have 
\begin{align*}
\twoFone[1-x]{c-b,d-b}{\delta}=x^{b-a}\twoFone[1-x]{d-a,c-a}{\delta}.
\end{align*}
Under the assumptions \eqref{parameters}, we have 
(see \cite[Eq.~2.21.1.11]{PrudnikovEtAlVol3} or \cite[Eq.~7.512.4]{GradshteynRyzhik}) 
\begin{equation}
\int_{0}^{1}x^{a+n-1}(1-x)^{\delta-1}\,\twoFone[1-x]{c-b,d-b}{\delta}\dx
=\frac{\Gamma(a+n)\Gamma(b+n)\Gamma(\delta)}{\Gamma(c+n)\Gamma(d+n)}.
\end{equation}
Therefore, $\dis\W(x;a,b;c,d)$ is a probability density function on the interval $(0,1)$ with moments 
\begin{align}
\label{moments of the hypergeometric weight}
\int_{0}^{1}x^n\W(x;a,b;c,d)\dx
=\frac{\pochhammer{a}\pochhammer{b}}{\pochhammer{c}\pochhammer{d}}
\quad \text{for} \quad n\in\N. 
\end{align}

Throughout the text, $\N=\Z_0^+=\{0,1,2,\cdots\}$.
When referring to $\polyseq$ as a polynomial sequence it is assumed that $P_n$ is a polynomial of a single variable with degree exactly $n$.
We consistently deal with monic polynomials, unless stated otherwise.

{\it Multiple orthogonal polynomials} are a generalisation of (standard) orthogonal polynomials.
We give a brief introduction to this topic here, further information can be found for instance in \cite[Ch.~23]{IsmailBook} and \cite{GuillermoSurvey}.

The orthogonality conditions of multiple orthogonal polynomials are spread across a vector of $r\in\Z^+$ measures and they are polynomials on a single variable depending on a multi-index $\n=(n_0,\cdots,n_{r-1})\in\N^r$ of length $|\n|=n_0+\cdots+n_{r-1}$.
There are two types of multiple orthogonal polynomials with respect to a system of $r$ measures $(\mu_0,\cdots,\mu_{r-1})$. 
When the number of measures is $r=1$, both types of multiple orthogonality reduce to standard orthogonality.
A polynomial sequence $\polyseq$ is orthogonal with respect to a measure $\mu$ if
\begin{align}
\label{standard orthogonality conditions}
\int x^kP_n(x)\dd\mu(x)
=\begin{cases}
0, &\text{ if } 0\leq k\leq n-1,\\
N_n\neq 0, &\text{ if } n=k.
\end{cases}
\end{align}

We focus on the case of $r=2$ measures but the definitions presented here are easily generalised for $r\geq 2$.

The \textit{type I multiple orthogonal polynomials} for $\n=(n_0,n_1)\in\N^2$ are given by a vector of $2$ polynomials $\left(A_{(n_0,n_1)},B_{(n_0,n_1)}\right)$, with $\deg A_{(n_0,n_1)}\leq n_0-1$ and $\deg B_{(n_0,n_1)}\leq n_1-1$, satisfying the orthogonality and normalisation conditions
\begin{align}
\label{orthogonality conditions type I}
\int x^kA_{(n_0,n_1)}(x)\dd\mu_0(x)+\int x^kB_{(n_0,n_1)}(x)\dd\mu_1(x)=
\begin{cases}
0, &\text{ if } 0\leq k\leq n_0+n_1-2,\\
1, &\text{ if } k=n_0+n_1-1.
\end{cases}
\end{align}
If the measures $\mu_0(x)$ and $\mu_1(x)$ are absolutely continuous with respect to a common positive measure $\mu$, that is, if there exist weight functions $w_0(x)$ and $w_1(x)$ such that $\dis\dd\mu_j(x)=w_j(x)\dd\mu(x)$, for both $j\in\{0,1\}$, then the \textit{type I function} is
\begin{align}
\label{type I function definition}
Q_{(n_0,n_1)}(x)=A_{(n_0,n_1)}(x)w_0(x)+B_{(n_0,n_1)}(x)w_1(x)
\end{align}
and the conditions in \eqref{orthogonality conditions type I} become
\begin{align}
\label{orthogonality conditions type I function}
\int x^kQ_{(n_0,n_1)}(x)\dd\mu(x)
=\begin{cases}
0, &\text{ if } 0\leq k\leq n_0+n_1-2,\\
1, &\text{ if } k=n_0+n_1-1.
\end{cases}
\end{align}
The \textit{type II multiple orthogonal polynomial} for $\n=(n_0,n_1)\in\N^2$ is a monic polynomial $P_{(n_0,n_1)}$ of degree $n_0+n_1$ which satisfies, for both $j\in\{0,1\}$, the orthogonality conditions
\begin{align}
\label{orthogonality conditions type II}
\int x^kP_{(n_0,n_1)}(x)\dd\mu_j(x)=0,
\;\;\;0\leq k\leq n_j-1.
\end{align}
The orthogonality conditions for both type I and type II multiple orthogonal polynomials give a non-homogeneous system of $n_0+n_1$ linear equations for the $n_0+n_1$ unknown coefficients of the vector of polynomials $\left(A_{(n_0,n_1)},B_{(n_0,n_1)}\right)$ in \eqref{orthogonality conditions type I} or the polynomials $P_{(n_0,n_1)}(x)$ in \eqref{orthogonality conditions type II}. 
If the solution exists, it is unique and the corresponding matrices of the system for type I and type II are the transpose to each other. 
However it is possible that this system doesn't have a solution, unless further conditions are imposed
(unlike standard orthogonality on the real line, the existence of such solutions is not a trivial matter). 
If there is a unique solution, then the multi-index $\n$ is called \textit{normal} and if all multi-indices are normal, the system is a \textit{perfect system}. 

An example of systems known to be perfect are the Algebraic Tchebyshev systems, or simply \textit{AT-systems} (see \cite[Ch.~4]{NikishinSorokinBook}).
A pair of measures $\dis\left(\mu_0,\mu_1\right)$ is an AT-system on an interval $I$ for a multi-index $\n=(n_0,n_1)\in\N^2$ if the measures $\mu_0(x)$ and $\mu_1(x)$ are absolutely continuous with respect to a common positive measure $\mu$ on $I$, via weight functions $w_0(x)$ and $w_1(x)$, and the set of functions
\begin{equation*}
\left\{w_0(x),xw_0(x),\cdots,x^{n_0-1}w_0(x),w_1(x),xw_1(x),\cdots,x^{n_1-1}w_1(x)\right\}
\end{equation*}
forms a Chebyshev system on $I$, meaning that for any polynomials $p_0$ and $p_1$ of degree not greater than $n_0-1$ and $n_1-1$, respectively, and not simultaneously equal to $0$, the function $\dis p_0(x)w_0(x)+p_1(x)w_1(x)$ has at most $n_0+n_1-1$ zeros on $I$.
A vector of measures $\dis(\mu_0,\mu_1)$ is an AT-system on an interval $I$ if it is an AT-system on $I$ for every multi-index in $\N^2$.

Another special example of a perfect system is a \textit{Nikishin system} (firstly introduced in \cite{NikishinSystems}). 
A pair of measures $(\mu_0,\mu_1)$ forms a Nikishin system (of order $2$) if both measures are supported on an interval $I_0$ and there exists a positive measure $\sigma$ on an interval $I_1$ with $I_0\cap I_1=\emptyset$ such that
\begin{align}
\label{Nikishin system - ratio of the measures}
\frac{\dd\mu_1(x)}{\dd\mu_0(x)}=\int_{I_1}\frac{\dd\sigma(t)}{x-t}.
\end{align}
It was proved in \cite{NikishinSystemsArePerfect} that every Nikishin system is perfect (see also \cite{NikishinSystemsArePerfectCaseOfUnboundedAndTouchingSupports} for the cases where the supports of the measures are unbounded or where consecutive intervals touch at one point).
More precisely, it is proved in \cite{NikishinSystemsArePerfect} and \cite{NikishinSystemsArePerfectCaseOfUnboundedAndTouchingSupports} that every Nikishin system is an AT-system, therefore it is perfect. 
Moreover, for any $(n_0,n_1)\in\N^2$ belonging to an AT-system on an interval $I$, the type I function for $Q_{(n_0,n_1)}$ defined by \eqref{type I function definition} has exactly $n_0+n_1-1$ sign changes on $I$ and the type II multiple orthogonal polynomial $P_{(n_0,n_1)}$ has $n_0+n_1$ simple zeros on $I$ which satisfy an interlacing property as there is always a zero of $P_{(n_0,n_1)}$ between two consecutive zeros of $P_{(n_0+1,n_1)}$ or $P_{(n_0,n_1+1)}$.
As a Nikishin system is always an AT-system, the same properties hold for Nikishin systems.

The main contribution of this paper is on multi-indices on the step line. 
A multi-index $(n_0,n_1)\in\N^2$ is on the step line if either $n_0=n_1$ or $n_0=n_1+1$ (alternatively to the latter we could consider $n_1=n_0+1$, that change is equivalent to swapping the roles of the measures).
For each $n\in\N$, there is a unique multi-index of length $n$ on the step line of $\N^2$.
More precisely, the multi-index of length $n$ is $\dis\n=(m,m)$, if $n=2m$, or $\dis\n=(m+1,m)$, if $n=2m+1$.
Therefore, when we only consider multi-indices on the step line, we can replace any multi-index by its length without any ambiguity.

For the type II multiple orthogonal polynomials on the step line, we obtain a polynomial sequence with exactly one polynomial of degree $n$ for each $n\in\N$.
These are often referred to as \textit{$d$-orthogonal polynomials} (where $d$ is the number of orthogonality measures), as introduced in \cite{MaroniOrthogonalite}. 
In the case of $d=2$ measures, the  type II multiple orthogonality conditions \eqref{orthogonality conditions type II} on the step line correspond to say that if we set 
\begin{equation*}
P_{2m}(x) = P_{m,m}(x) 
\quad \text{and}\quad 
P_{2m+1}(x) = P_{m+1,m}(x), 
\end{equation*}
then the polynomial sequence $\dis\polyseq$ is \textit{$2$-orthogonal} with respect to a pair of measures $(\mu_0,\mu_1)$ if 
for each $j\in\{0,1\}$ 
\begin{align}
\label{2-orthogonality conditions}
\int x^kP_n(x)\dd\mu_j(x)
=\begin{cases}
0, &\text{if } n\geq 2k+j+1,\\
N_n\neq 0, &\text{if } n=2k+j.
\end{cases}
\end{align}

Straightforwardly from the definition \eqref{orthogonality conditions type II} observe that $\{P_{(n,0)}\}_{n\geq 0}$ and $\{P_{(0,n)}\}_{n\geq 0}$ are (standard) orthogonal polynomial sequences with respect to the measures $\mu_0$ and $\mu_1$, respectively. As such,   by the spectral theorem for orthogonal polynomials ({\it aka} Shohat-Favard theorem)
: $\{P_{(n,0)}\}_{n\geq 0}$ and $\{P_{(0,n)}\}_{n\geq 0}$ are orthogonal if and only if there exist coefficient two pairs of coefficients $(\beta_{n}^{(0)},\gamma_{n}^{(0)}) $ and 
$( \beta_{n}^{(1)},\gamma_{n}^{(1)})$ with $\gamma_{n}^{(j)}\neq 0$ for all $n\geq 1$ and each $j=1,2$ such that 
 $\{P_{(n,0)}\}_{n\geq 0}$ and $\{P_{(0,n)}\}_{n\geq 0}$ respectively satisfy the second order recurrence relation 
\begin{align*}
p_{n+1}(x)=(x- \beta_{n}^{(j)})p_n(x)- \gamma_{n}^{(1)}p_{n-1}(x),
\end{align*}
with initial conditions $p_{-1}=0$ and $p_0=1$. 
Moreover, if the $\beta$-coefficients are all real and the $\gamma$-coefficients are all positive, then $\mu$ is a positive measure on the real line.

Multiple orthogonal polynomials also satisfy (nearest-neighbour) recurrence relations (see \cite{WalterNearestNeighborRecurrenceRelations}). 
In particular, when the indexes lie on the step line, a polynomial sequence $\polyseq$ is $2$-orthogonal if and only if it satisfies a third order recurrence relation of the type
\begin{align}
\label{recurrence relation for a 2-OPS}
P_{n+1}(x)=(x-\beta_n)P_n(x)-\alpha_{n}P_{n-1}(x)-\gamma_{n-1}P_{n-2}(x),
\end{align}
with $\gamma_n\neq 0$, for all $n\geq 1$, and initial conditions $P_{-2}=P_{-1}=0$ and $P_0=1$.

The latter recurrence relation can be expressed, for each $n\in\Z^+$, as 
\begin{align}
\label{recurrence relation for a 2-OPS in matrix form}
\mathrm{H}_n\begin{bmatrix} P_0(x) \\ P_1(x) \\  \vdots \\ P_{n-2}(x) \\ P_{n-1}(x) \end{bmatrix}
= x\begin{bmatrix} P_0(x) \\ P_1(x) \\  \vdots \\ P_{n-2}(x) \\ P_{n-1}(x) \end{bmatrix}
- P_n(x) \begin{bmatrix} 0 \\ 0 \\  \vdots \\ 0 \\ 1 \end{bmatrix},
\end{align}
involving the truncated lower-Hessenberg matrix
\begin{align}
\label{Hessenberg matrix}
\mathrm{H}_n=\begin{bmatrix} 
\beta_0 & 1 & 0 & 0 & \cdots & 0 \\
\alpha_1 & \beta_1 & 1 & 0 & \cdots & 0 \\ 
\gamma_1 & \alpha_2 & \beta_2 & 1 & \ddots & \vdots \\
0 & \ddots & \ddots & \ddots & \ddots & 0 \\
\vdots & \ddots & \gamma_{n-3} & \alpha_{n-2} & \beta_{n-2} & 1\\
0 & \cdots & 0 & \gamma_{n-2} & \alpha_{n-1} & \beta_{n-1}
\end{bmatrix}.
\end{align}
Therefore, the zeros of $P_n(x)$ correspond to the eigenvalues of the Hessenberg matrix $\mathrm{H}_n$, which highlights  the connection between multiple orthogonal polynomials and the spectral theory of non-selfadjoint operators explored in \cite{AptekarevKaliaguine1998} and \cite{WalterNonsymmetric}, among others.

For the type I multiple orthogonal polynomials on the step line for $r=2$ measures, we have 
\begin{align*}
\deg(A_n)\leq\floor{\frac{n-1}{2}} 
\quad\text{and}\quad 
\deg(B_n)\leq\floor{\frac{n}{2}}-1,
\end{align*}
that is, $\deg(A_n)=m-1$, if $n=2m$ or $n=2m-1$, and $\deg(B_n)=m-1$, if $n=2m$ or $n=2m+1$. 
Assuming that there exists a positive measure $\mu$ and a pair of weight functions $\dis(w_0,w_1)$ such that $\dis\dd\mu_0(x)=w_0(x)\dd\mu(x)$ and $\dis\dd\mu_1(x)=w_1(x)\dd\mu(x)$, the type I function on the step line is
\begin{align}\label{typeI An Bn}
Q_n(x)=A_n(x)w_0(x)+B_n(x)w_1(x)
\end{align}
and the orthogonality and normalisation conditions correspond to
\begin{align}
\int x^kQ_n(x)\dd\mu(x)
=\begin{cases}
0, &\text{ if } 0\leq k\leq n-2,\\
1, &\text{ if } k=n-1.
\end{cases}
\end{align}
Further information about multiple orthogonal polynomials can be found for instance in \cite[Ch.~23]{IsmailBook} and \cite{GuillermoSurvey}.

We start Section \ref{Multiple orthogonality associated with the hypergeometric weights} by showing that the weight functions $\W(x;a,b;c,d)$ and $\W(x;a,b+1;c+1,d)$ in \eqref{hypergeometric weight definition} form a Nikishin system (see Theorem \ref{Nikishin system}).
This readily implies that the multiple orthogonal polynomials of both type II and type I with respect to these weight functions exist and are unique for every multi-index $\n=(n_0,n_1)\in\N^2$ and their zeros satisfy the properties as those of an AT-system. 
Next we obtain a second order differential equation and a matrix differential equation satisfied by the weight functions (see Theorems \ref{differential equation satisfied by the hypergeometric weight - theorem} and \ref{canonical matrix differential equation satisfied by the hypergeometric weights - theorem}, respectively), which we use to deduce differential properties for the multiple orthogonal polynomials of both type II and type I on the step line (see Theorem \ref{differentiation formulas for the multiple orthogonal polynomials - theorem - hypergeometric weights}).
More precisely, we show that the differentiation of both type II and type I polynomials on the step line gives a shift on the parameters as well as on the index.
So this means that these multiple orthogonal polynomials satisfy the so called {\it Hahn's property}: the sequence of its derivatives is again multiple orthogonal.
In particular, the type II polynomials stand as an example of a Hahn-classical $2$-orthogonal family.
Finally, we derive a Rodrigues-type formula for the type I functions on the step line (see Theorem \ref{Rodrigues formula for the type I polynomials - theorem - hypergeometric weights}) as well as a recursive relation generating the type I polynomials.

Section \ref{Type II MOPs wrt hypergeometric weights} is devoted to the characterisation of the $2$-orthogonal polynomials with respect to the pair of weights $\dis\big[\W(x;a,b;c,d),\W(x;a,b+1;c+1,d)\big]$. To begin with, in \S \ref{explicit expression}, we give an explicit expression for these polynomials as terminating generalised hypergeometric series, more precisely as $\dis{}_{3}F_{2}$. 
Generalised hypergeometric series are formally defined by
\begin{align}
\label{generalised hypergeometric series}
\Hypergeometric[z]{p}{q}{\alpha_1,\cdots,\alpha_p}{\beta_1,\cdots,\beta_q}
=\sum_{n=0}^{\infty}
\frac{\pochhammer{\alpha_1}\cdots\pochhammer{\alpha_p}}
{\pochhammer{\beta_1}\cdots\pochhammer{\beta_q}}
\frac{z^n}{n!} \ ,
\end{align}
where $p,q\in\N$, $z,\alpha_1,\cdots,\alpha_p\in\C$ and $\beta_1,\cdots,\beta_p\in\C\backslash\{-n : \ n\in\N\}$.
If one of the parameters $\alpha_1,\cdots,\alpha_p$ is a non-positive integer, the series \eqref{generalised hypergeometric series} terminates and defines a (hypergeometric type) polynomial. 
When the series does not terminate, it converges for all finite values of $z$ if $p\leq q$ and on the open unit disk $|z|<1$ (with convergence on the unit circle depending on the parameters) if $p=q+1$ and it diverges for any $z\neq 0$ otherwise.
When the series is convergent, the function defined by \eqref{generalised hypergeometric series} is a solution to the generalised hypergeometric differential equation (see \cite[Eq.~16.8.3]{DLMF})
\begin{align}
\label{generalised hypergeometric differential equation}
\left[\left(z\,\DiffOp[z]+\beta_1\right)\cdots\left(z\,\DiffOp[z]+\beta_q\right)\DiffOp[z]\right]F(z)
=\left[\left(z\,\DiffOp[z]+\alpha_1\right)\cdots\left(z\,\DiffOp[z]+\alpha_p\right)\right]F(z).
\end{align}

Note that the latter reduces to \eqref{hypergeometric differential equation} when $(p,q)=(2,1)$. Thus, based on the explicit expression for the $2$-orthogonal polynomials we are able to describe them as a solution to a third order differential equation (of hypergeometric type) in \S\ref{differential equation} and in \S\ref{recurrence relation} as a solution to a third order recurrence relation. 
Particular choices on the parameters $a,b,c,d$ of these polynomials result in known multiple orthogonal polynomials. So, in  \S \ref{Jacobi-type 2-OPS and sequence with constant recurrence relation coefficients} we make the connection to the so-called Jacobi-type $2$-orthogonal polynomials investigated in \cite{LamiriOuni}, where  we pay particular attention to the case where all the coefficients are constant. 
The recurrence relation coefficients of the multiple orthogonal polynomials under analysis can be written as combinations of the coefficients of a branched continued fraction representation for a generalised hypergeometric function derived in \cite{AlanSokalEtAlBranchedContinuedFractions}.
Hence, these recurrence coefficients are real, positive and bounded, whose asymptotic behaviour coincides with the one of the recurrence relation coefficients of Jacobi-Pi\~neiro (type II) multiple orthogonal polynomials on the step line studied in \cite{WalterCoussementSomeClassicalMOPs}. 
As a consequence (see \cite{WalterCoussementX2AsymptoticZeroDistribution}), the two distinct polynomial sets also share the same ratio asymptotics and therefore the same asymptotic zero distribution as well as the same Mehler-Heine asymptotic near the endpoint at $0$, as detailed in \S \ref{Jacobi-Pineiro polynomials and asymptotic results}. 
Besides, other particular choices on the parameters $a,b,c,d$ lead to the three components of certain $3$-fold symmetric Hahn-classical $2$-orthogonal polynomials on star-like sets that appeared in \cite{AnaWalter3FoldSym}, as we explain in  \S \ref{Hahn-classical 3-fold symmetric 2-OPS}. 
Finally, we establish confluence relations (or limiting relations on the parameters) to other Hahn-classical $2$-orthogonal polynomials of hypergeometric type, such as the ones investigated in \cite{PaperTricomiWeights}.

\section{Differential properties and multiple orthogonality}
\label{Multiple orthogonality associated with the hypergeometric weights}

This investigation starts with a pair of weight functions $\W(x;a,b;c,d)$ and $\W(x;a,b+1;c+1,d)$ defined in \eqref{hypergeometric weight definition} subject to the constraints \eqref{parameters} on the parameters $a,b,c$ and $d$. The goal is to describe multiple orthogonal polynomials of type I and type II with respect to these weights. Before doing so, we aim to prove that such sets of polynomials exist and are unique. A fact that is proved to be true, after it is shown in Theorem \ref{Nikishin system} in \S \ref{Nikishin system} that the vector of weights $\dis\big[\W(x;a,b+1;c+1,d),\W(x;a,b;c,d)\big]$ forms a Nikishin system. The characterisation of the polynomials is guided by the algebraic and differential properties of the weights. As such, the technical results described in Theorem \ref{canonical matrix differential equation satisfied by the hypergeometric weights - theorem} (regarding the vector of weights) and Theorem \ref{differentiation formulas for the multiple orthogonal polynomials - theorem - hypergeometric weights} (for the differential properties of the polynomials of type II and the type I functions) form the basis of an explicit analysis carried on in \S \ref{Type I multiple orthogonal polynomials} for the type I and in the next Section \ref{Type II MOPs wrt hypergeometric weights} for the type II polynomials. 
 
\subsection{Nikishin system}
We show that the weight vector $\dis\big[\W(x;a,b+1;c+1,d),\W(x;a,b;c,d)\big]$ forms a Nikishin system, which guarantees that both type I and II multiple orthogonal polynomials with respect to these weight functions exist and are unique for every multi-index $\dis(n_0,n_1)\in\N^2$ as well as it implies that the type I multiple orthogonal polynomials $A_{(n_0,n_1)}$ and $B_{(n_0,n_1)}$ have degree exactly $n_0-1$ and $n_1-1$, respectively, and the type II multiple orthogonal polynomial $P_{(n_0,n_1)}$ has $n_0+n_1$ positive real simple zeros that satisfy the usual interlacing property: there is always a zero of $P_{(n_0,n_1)}$ between two consecutive zeros of $P_{(n_0+1,n_1)}$ or $P_{(n_0,n_1+1)}$.

To prove this result, we use the connection between continued fractions and Stieltjes transforms to guarantee the existence of an integral representation of the type in \eqref{Nikishin system - ratio of the measures} for the ratio of the weight functions involved. For simplicity, we follow the notation for continued fractions used in \cite{ContinuedFractionsForSpecialFunctions}: 
\begin{align}
\label{continued fraction}
\contfrac[n=0]{a_n}{b_n}:=\cfrac{a_0}{b_0+\cfrac{a_1}{b_1+\cfrac{a_2}{b_2+\cdots}}}.
\end{align}
Particularly relevant to this work are the so-called {\it Stieltjes continued fractions} or, simply, {\it S-fractions}, due to their connection with Stieltjes transforms which was firstly investigated in \cite{Stieltjesmemoir}.
The continued fraction playing a role here is an example of a modified S-fraction which is obtained if, for some constants $\alpha_k$, $k\in\N$, we set in \eqref{continued fraction}, $a_0=\alpha_0$ and, for any $n\in\N$, $b_n=1$ and $a_{n+1}=\alpha_{n+1}z$, to obtain 
\begin{align}
\label{modified S-fraction}
F(z)=\cfrac{\alpha_0}{1+\cfrac{\alpha_1 z}{1+\cfrac{\alpha_2 z}{1+\cdots}}} .
\end{align}

The main result of this subsection is the following.
\begin{theorem}
\label{Nikishin system}
Let $\dis\W(x;a,b;c,d)$ be given \eqref{hypergeometric weight definition} under the assumptions \eqref{parameters}. 
The ratio 
\[\frac{\W(x;a,b;c,d)}{\W(x;a,b+1;c+1,d)}\] can be represented via the continued fraction  \eqref{modified S-fraction} with $z=x-1$ and 
$\dis\alpha_n=\left(1-g_{n-1}\right)g_n$, where $g_0=0$, $\dis g_{2k+1}=\frac{c-b+k}{\delta+2k}$ and $\dis g_{2k+2}=\frac{d-b+k}{\delta+2k+1}$ for $n\geq 1$ and $k  \geq 0$.
Moreover, there exist probability density functions $\sigma$ in $(0,1)$ and $\theta$ in $(1,+\infty)$ such that
\begin{align}
\label{ratio of the hypergeometric weights - integral representations}
\frac{\W(x;a,b;c,d)}{\W(x;a,b+1;c+1,d)}
=\frac{c}{b}\int_{0}^{1}\frac{\dd\sigma(t)}{1+t(x-1)}
=\frac{c}{b}\int_{-\infty}^{-1}\frac{\dd\theta(-u)}{x-1-u}
=\frac{c}{b}\int_{-\infty}^{0}\frac{\dd\theta(1-s)}{x-s}.
\end{align}
Therefore, the vector of weight functions $\dis\big[{\W(x;a,b+1;c+1,d)},{\W(x;a,b;c,d)}\big]$ forms a Nikishin system on the interval $(0,1)$.
\end{theorem}

\begin{proof}
Recalling \eqref{hypergeometric weight definition},
\begin{equation}
\label{ratio of hypergeometric weights}
\frac{\W(x;a,b;c,d)}{\W(x;a,b+1;c+1,d)}
=\frac{c}{b}\,\frac{\twoFoneOneLine[1-x]{c-b,d-b}{\delta}}{\twoFoneOneLine[1-x]{c-b,d-b-1}{\delta}},
\quad \text{with } \delta=c+d-a-b.
\end{equation}
Therefore, the ratio of weight functions above admits a representation similar to Gauss' continued fraction.
More precisely, based on \cite[Eq.~(14.29)]{AlanSokalEtAlBranchedContinuedFractions}, the ratio of weights in \eqref{ratio of hypergeometric weights} can be represented by a continued fraction of the type in \eqref{modified S-fraction}, with $z=x-1$, and coefficients
\begin{align*}
\alpha_0=\frac{c}{b};\;
\alpha_1=\frac{c-b}{\delta};\;
\alpha_{2k+1}=\frac{(c-b+k)(c-a+k)}{(\delta+2k-1)(\delta+2k)},\,k\geq 1;\;
\dis\alpha_{2k+2}=\frac{(d-b+k)(d-a+k)}{(\delta+2k)(\delta+2k+1)},\,k\in\N.
\end{align*}
Moreover, the coefficients $\dis\alpha_n$, $n\geq 1$, can be rewritten as $\dis\alpha_n=\left(1-g_{n-1}\right)g_n$, \vspace*{0,1 cm}
with $g_0=0$ and, for each $k\in\N$, $\dis g_{2k+1}=\frac{c-b+k}{\delta+2k}$ and $\dis g_{2k+2}=\frac{d-b+k}{\delta+2k+1}$ 
(see \cite[Eqs.~(2.7)-(2.8)]{Kustner}). \vspace*{0,1 cm}
Note that $0<g_n<1$, for all $n\geq 1$, and, as a result, the continued fraction described above is of the type in \cite[Eq.~27.8]{WallContinuedFractions}. 
Therefore, the first integral representation in \eqref{ratio of the hypergeometric weights - integral representations} can be derived directly from \cite[Eq.~67.5]{WallContinuedFractions} and the second one can be deduced combining \cite[Ths.~67.1~\&~27.5]{WallContinuedFractions}, while the last equality  in \eqref{ratio of the hypergeometric weights - integral representations} is obtained via the change of variable $s=u+1$.
\end{proof}

Under the additional assumption $b>a-1$ and using a recent result from Dyachenko and Karp in \cite{DyachenkoKarp}, the generating measure $\sigma$ in \eqref{ratio of the hypergeometric weights - integral representations} admits the following integral representation 
\begin{subequations}
\begin{align}
\label{W over W+}
\hspace{-.4cm}\frac{\W(x;a,b;c,d)}{\W(x;a,b+1;c+1,d)}
=\int_{0}^{1}\hspace{-.1cm}\frac{\lambda t^{c+d-2b-2}(1-t)^{b-a}\mathrm{dt}}{\left(1+t(x-1)\right)\left|\twoFoneOneLine[t^{-1}]{c-b,d-b-1}{\delta}\right|^2}+K(c,d)
\end{align}
with
\begin{align*}
\lambda=\frac{c\left(\Gamma(\delta)\right)^2}{b\Gamma(c-b)\Gamma(d-b)\Gamma(d-a)\Gamma(c-a+1)}
\;\text{ and }\;
K(c,d)=\begin{cases}
0, & \text{ if } d\leq c+1,\\
\frac{d-c-1}{d-1}, & \text{ if } d\geq c+1.
\end{cases}
\end{align*}
The change of variable $t= \frac{1}{1-s}$ in \eqref{W over W+} gives
\begin{align}\label{W over W+2}
\frac{\W(x;a,b;c,d)}{\W(x;a,b+1;c+1,d)}
=\int_{-\infty}^{0}\frac{\lambda(-s)^{b-a}(1-s)^{1-\delta}}{(x-s)\Big|\twoFoneOneLine[1-s]{c-b,d-b-1}{\delta}\Big|^2}\mathrm{d}s+K(c,d).
\end{align}
\end{subequations}
Hence, if $b>a-1$, the measures in the first and last integral representations in \eqref{ratio of the hypergeometric weights - integral representations} can be explicitly represented by \eqref{W over W+} and \eqref{W over W+2}, respectively.

\subsection{Differential properties}

We start by describing the weight function $\W(x;a,b;c,d)$ in \eqref{hypergeometric weight definition} as a solution to a second-order ordinary differential equation, to then describe the vector of weight functions 
\begin{equation}
\label{hypergeometric weight vector}
	\Wvec(x;a,b;c,d):=\twovector{\W(x;a,b;c,d)}{\W(x;a,b+1;c+1,d)}, 
\end{equation}
as a solution to a system of first order differential equations in Theorem \ref{canonical matrix differential equation satisfied by the hypergeometric weights - theorem}. A result that is crucial to obtain, in Theorem \ref{differentiation formulas for the multiple orthogonal polynomials - theorem - hypergeometric weights}, differential properties on the system of multiple orthogonal polynomials of type II and the functions of type I, revealing their Hahn-classical property.

\begin{proposition}
\label{differential equation satisfied by the hypergeometric weight - theorem}
For $a,b,c,d\in\R^+$ such that $\min\{c,d\}>\max\{a,b\}$, let $\dis\W(x):=\W(x;a,b;c,d)$ be the weight function defined by \eqref{hypergeometric weight definition}. Then
\begin{align}
\label{ODE satisfied by the hypergeometric weight}
(1-x)x^2\W''(x)+\left((c+d-5)x-(a+b-3)\right)x\,\W'(x)+\left((a-1)(b-1)-(c-2)(d-2)x\right)\W(x)=0.
\end{align}
\end{proposition}

\begin{proof}
We set $\dis\delta=c+d-a-b$, $\dis\lambda=\frac{\Gamma(c)\Gamma(d)}{\Gamma(a)\Gamma(b)\Gamma(\delta)}$ and $\dis F(z)=\twoFoneOneLine[z]{c-b,d-b}{\delta}$ so that
\begin{align*}
\W(x)=\lambda x^{a-1}(1-x)^{\delta-1}F(1-x).
\end{align*}
Differentiating this expression twice, we get
\begin{align}
\label{W,W',W''}
\W^\roundparameter{j}(x)=\lambda x^{a-1-j}(1-x)^{\delta-1-j}F_j(x),
\quad j=0,1,2,
\end{align}
with $\dis F_0(x)=F(1-x)$,
\begin{align*}
F_1(x)
=\big((2+b-c-d)x+(a-1)\big)F(1-x)-x(1-x)F'(1-x)
\end{align*}
and
\begin{align*}
F_2(x)
&=x^2(1-x)^2F''(1-x)+2x(1-x)\Big((c+d-b-2)x+(1-a)\Big)F'(1-x) \\&
+\Big((c+d-b-2)(c+d-b-3)x^2-2(a-1)(c+d-b-3)x+(a-1)(a-2)\Big)F(1-x).
\end{align*}
Recall  \eqref{hypergeometric differential equation} 
to derive that  $\dis F(1-x)=\twoFoneOneLine[1-x]{c-b,d-b}{\delta}$  satisfies 
\begin{align}
x(1-x)F''(1-x)=(c-b)(d-b)F(1-x)-((c+d-2b+1)x+(b-a-1))F'(1-x),
\end{align}
which can be used to rewrite $F_2(x)$ as
\begin{align*}
F_2(x)=\left((5-c-d)x+(a+b-3)\right)F_1(x)+(1-x)\left((c-2)(d-2)x-(a-1)(b-1)\right)F_0(x).
\end{align*}
Combining the latter relation with \eqref{W,W',W''}, we derive \eqref{ODE satisfied by the hypergeometric weight}.
\end{proof}

Based on the second order differential equation \eqref{ODE satisfied by the hypergeometric weight} we deduce a system of first order differential equations for which the vector \eqref{hypergeometric weight vector} is a solution.   \\

\begin{theorem}
\label{canonical matrix differential equation satisfied by the hypergeometric weights - theorem}
Let $\Wvec(x;a,b;c,d)$ as defined in \eqref{hypergeometric weight vector} subject to \eqref{parameters}. Then, the following identities hold 
\begin{equation}
\label{shift on the parameters - hypergeometric weights}
x\,\Phi(x)\Wvec(x;a,b;c,d)=\Wvec(x;a+1,b+1;d+1,c+2)
\end{equation}
and
\begin{equation}
\label{canonical matrix differential equation satisfied by the hypergeometric weights}
\DiffOp\Big(x\,\Phi(x)\Wvec(x;a,b;c,d)\Big)
+\Psi(x)\Wvec(x;a,b;c,d)=0,
\end{equation}
where
\begin{align}
\label{matrix Phi in matrix diff eq}
\Phi(x):=\Phi(x;a,b;c,d)=
\matrixtwobytwo{\frac{c(c+1)d}{ab(c-b)}}{-\frac{(c+1)d}{a(c-b)}}
{-\frac{c(c+1)d(d+1)}{ab(b+1)(d-a)}\,x}{\frac{(c+1)d(d+1)}{a(b+1)(d-a)}}
\end{align}
and
\begin{align}
\label{matrix Psi in matrix diff eq}
\Psi(x):=\Psi(x;a,b;c,d)=
\matrixtwobytwo{-\frac{c(c+1)d}{a(c-b)}}{\frac{c(c+1)d}{a(c-b)}}
{\frac{c(c+1)d^2(d+1)}{ab(b+1)(d-a)}\,x}{-\frac{(c+1)d(d+1)}{(b+1)(d-a)}}.
\end{align}

\end{theorem}
\begin{proof}[Proof of Theorem \ref{canonical matrix differential equation satisfied by the hypergeometric weights - theorem}.] 
In order to prove \eqref{shift on the parameters - hypergeometric weights}, we need to check that 
\begin{align*}
\twovector{\V_0(x)}{\V_1(x)}
:=x\,\Phi(x)\twovector{\W(x;a,b;c,d)}{\W(x;a,b+1;c+1,d)}
=\twovector{\W(x;a+1,b+1;d+1,c+2)}{\W(x;a+1,b+2;d+2,c+2)}.
\end{align*}

Firstly,
\begin{align*}
\V_0(x) &
=x\,\frac{(c+1)d}{a(c-b)}\left(\frac{c}{b}\,\W(x;a,b;c,d)-\W(x;a,b+1;c+1,d)\right) \\&
=\frac{\Gamma(c+2)\Gamma(d+1)x^a(1-x)^{\delta-1}}{\Gamma(a+1)\Gamma(b+1)\Gamma(\delta)(c-b)}
\left(\twoFone[1-x]{c-b,d-b}{\delta}-\twoFone[1-x]{c-b,d-b-1}{\delta}\right).
\end{align*}

Based on \cite[Eq.~15.5.15 \& Eq.~15.5.16]{DLMF}, we obtain, respectively,
\begin{align*}
(c-b)\,\twoFone[1-x]{c-b+1,d-b}{\delta+1}
=\delta\,\twoFone[1-x]{c-b,d-b}{\delta}-(d-a)\,\twoFone[1-x]{c-b,d-b}{\delta+1}
\end{align*}
and
\begin{align*}
\frac{d-a}{\delta}(1-x)\twoFone[1-x]{c-b,d-b}{\delta+1}
=\twoFone[1-x]{c-b,d-b-1}{\delta}-x\,\twoFone[1-x]{c-b,d-b}{\delta}.
\end{align*}

Therefore, we can derive that
\begin{align*}
\frac{c-b}{\delta}\,(1-x)\,\twoFone[1-x]{c-b+1,d-b}{\delta+1}
=\twoFone[1-x]{c-b,d-b}{\delta}-\twoFone[1-x]{c-b,d-b-1}{\delta}
\end{align*}
and, as a result,
\begin{align}
\label{V0}
\V_0(x)
=\frac{\Gamma(c+2)\Gamma(d+1)x^a(1-x)^{\delta}}{\Gamma(a+1)\Gamma(b+1)\Gamma(\delta+1)}\, \twoFone[1-x]{c-b+1,d-b}{\delta+1}
=\W(x;a+1,b+1;d+1,c+2).
\end{align}

Similarly, we have
\begin{align*}
\V_1(x) &
=x\,\frac{(c+1)d(d+1)}{a(b+1)(d-a)}\left(\W(x;a,b+1;c+1,d)-\frac{c}{b}\,x\,\W(x;a,b;c,d)\right) \\&
=\frac{\Gamma(c+2)\Gamma(d+2)x^a(1-x)^{\delta-1}}{\Gamma(a+1)\Gamma(b+2)\Gamma(\delta)(d-a)}
\left(\twoFone[1-x]{c-b,d-b-1}{\delta}-x\,\twoFone[1-x]{c-b,d-b}{\delta}\right)
\end{align*}
and
\begin{align*}
\twoFone[1-x]{c-b,d-b-1}{\delta}-x\,\twoFone[1-x]{c-b,d-b}{\delta}
=\frac{d-a}{\delta}(1-x)\twoFone[1-x]{c-b,d-b}{\delta+1},
\end{align*}
hence we get
\begin{align}
\label{V1}
\V_1(x)
=\frac{\Gamma(c+2)\Gamma(d+2)x^a(1-x)^{\delta}}{\Gamma(a+1)\Gamma(b+2)\Gamma(\delta+1)} \twoFone[1-x]{d-b,c-b}{\delta+1}
=\W(x;a+1,b+2;d+2,c+2).
\end{align}

In order to prove \eqref{canonical matrix differential equation satisfied by the hypergeometric weights}, we need to check that 
\begin{align*}
\twovector{\V'_0(x)}{\V'_1(x)}
=\twovector{\frac{c(c+1)d}{a(c-b)}\Big(\W(x;a,b;c,d)-\W(x;a,b+1;c+1,d)\Big)}
{\frac{(c+1)d(d+1)}{(b+1)(d-a)}\Big(\W(x;a,b+1;c+1,d)-\frac{cd}{ab}\,x\,\W(x;a,b;c,d)\Big)}.
\end{align*}

Recalling \eqref{V0},
\begin{align*}
\V'_0(x) &
=\W'(x;a+1,b+1;d+1,c+2)  \\&
=\frac{\Gamma(c+2)\Gamma(d+1)}{\Gamma(a+1)\Gamma(b+1)\Gamma(\delta+1)}
 \DiffOp\left(x^a(1-x)^{\delta}\,\twoFone[1-x]{c-b+1,d-b}{\delta+1}\right),
\end{align*}
which is equivalent to
\begin{align*}
\V'_0(x)=\frac{\Gamma(c+2)\Gamma(d+1)x^{a-1}(1-x)^{\delta-1}}{\Gamma(a+1)\Gamma(b+1)\Gamma(\delta+1)}\,G_0(x),
\end{align*}
with
\begin{align*}
G_0(x)= &
(a-(c+d-b)x)\twoFone[1-x]{c-b+1,d-b}{\delta+1} \\&
-\frac{(c-b+1)(d-b)}{\delta+1}x(1-x)\twoFone[1-x]{c-b+2,d-b+1}{\delta+2}.
\end{align*}

Using \cite[Eq.~15.5.19]{DLMF},
\begin{align}
\label{?}
\begin{split}
&\frac{(c-b+1)(d-b)}{\delta+1}x(1-x)\,\twoFone[1-x]{c-b+2,d-b+1}{\delta+2}\\
=&\,\delta\,\twoFone[1-x]{c-b,d-b-1}{\delta}
-\big((c+d-2b)x+(b-a)\big)\,\twoFone[1-x]{c-b+1,d-b}{\delta+1}.
\end{split}
\end{align}
so that
\begin{align*}
G_0(x) &
=-\delta\,\twoFone[1-x]{c-b,d-b-1}{\delta}+b(1-x)\,\twoFone[1-x]{c-b+1,d-b}{\delta+1} \\&
=\frac{\delta}{c-b}\left(b\,\twoFone[1-x]{c-b,d-b}{\delta}-c\,\twoFone[1-x]{c-b,d-b-1}{\delta+1}\right).
\end{align*}
Therefore,
\begin{align*}
\V'_0(x)
=\frac{\Gamma(c+2)\Gamma(d+1)x^{a-1}(1-x)^{\delta-1}}{\Gamma(a+1)\Gamma(b)\Gamma(\delta)(c-b)}
\left(\twoFone[1-x]{c-b,d-b}{\delta}-\frac{c}{b}\,\twoFone[1-x]{c-b,d-b-1}{\delta+1}\right),
\end{align*}
which implies that
\begin{align}
\label{V0'}
\V'_0(x)=\frac{c(c+1)d}{a(c-b)}\Big(\W(x;a,b;c,d)-\W(x;a,b+1;c+1,d)\Big).
\end{align}

Similarly,
\begin{align*}
\V'_1(x)=\W'(x;a+1,b+2;d+2,c+2)
=\frac{\Gamma(c+2)\Gamma(d+2)}{\Gamma(a+1)\Gamma(b+2)\Gamma(\delta+1)}
\DiffOp\left(x^a(1-x)^{\delta}\,\twoFone[1-x]{c-b,d-b}{\delta+1}\right),
\end{align*}
which is equivalent to
\begin{align*}
\V'_1(x)=\frac{\Gamma(c+2)\Gamma(d+1)x^{a-1}(1-x)^{\delta-1}}{\Gamma(a+1)\Gamma(b+1)\Gamma(\delta+1)}\,G_1(x),
\end{align*}
with
\begin{align*}
G_1(x) =&(a-(c+d-b)x)\twoFone[1-x]{c-b,d-b}{\delta+1} \\
&-\frac{(c-b)(d-b)}{\delta+1}x(1-x)\twoFone[1-x]{c-b+1,d-b+1}{\delta+2}.
\end{align*}

Using \eqref{?} with a shift $c\to c-1$ and $a\to a-1$, we derive that
\begin{align*}
G_1(x) &
=-\delta\,\twoFone[1-x]{c-b-1,d-b-1}{\delta}+(b+1)(1-x)\,\twoFone[1-x]{c-b,d-b}{\delta+1} \\&
=\frac{\delta}{d-a}\left(a\,\twoFone[1-x]{c-b,d-b-1}{\delta}-dx\,\twoFone[1-x]{c-b,d-b}{\delta}\right).
\end{align*}
Therefore,
\begin{align*}
\V'_1(x)=& \frac{\Gamma(c+2)\Gamma(d+2)x^{a-1}(1-x)^{\delta-1}}{\Gamma(a)\Gamma(b+2)\Gamma(\delta)(d-a)}
	\Big(\twoFone[1-x]{c-b,d-b}{\delta} \\ 
	&-\frac{d}{a}\,x\,\twoFone[1-x]{c-b,d-b-1}{\delta+1}\Big),
\end{align*}
which implies that
\begin{align}
\V'_1(x)=\frac{(c+1)d(d+1)}{(b+1)(d-a)}\Big(\W(x;a,b+1;c+1,d)-\frac{cd}{ab}\,x\,\W(x;a,b;c,d)\Big).
\end{align}
\end{proof}

The latter result  guarantees  the Hahn-classical property of the multiple orthogonal polynomials investigated here. 
In fact, combining it with \cite[Prop.~2.6~\&~2.7]{PaperTricomiWeights} we show that the differentiation with respect to the variable of both type I and type II polynomials on the step line gives a shift on the parameters as well as on the index, as detailed in the next result.

\begin{theorem}
\label{differentiation formulas for the multiple orthogonal polynomials - theorem - hypergeometric weights}
For $a,b,c,d\in\R^+$ such that $\min\{c,d\}>\max\{a,b\}$, let $\dis P_n(x;a,b;c,d)$ and $\dis Q_n(x;a,b;c,d)$, with $n\in\N$, be, respectively, the type II multiple orthogonal polynomial and the type I function for the index of length $n$ on the step line with respect to $\dis\Wvec(x;a,b;c,d)$.
Then
\begin{align}
\label{differentiation formula for type II polynomials - hypergeometric weights}
\DiffOp\left(P_{n+1}(x;a,b;c,d)\right)=(n+1)P_n(x;a+1,b+1;d+1,c+2)
\end{align}
and 
\begin{align}
\label{differentiation formula for type I polynomials - hypergeometric weights}
\DiffOp\big(Q_n(x;a+1,b+1;d+1,c+2)\big)=-n\,Q_{n+1}(x;a,b;c,d). 
\end{align}
\end{theorem}

\begin{proof}
Let $\dis\Phi(x)$ be defined by \eqref{matrix Phi in matrix diff eq} and denote $\dis\Wvec(x;a,b;c,d)$ by $\dis\Wvec(x)$.

Since $\dis\Wvec(x)$ satisfies the equation \eqref{canonical matrix differential equation satisfied by the hypergeometric weights} and on account of the degrees of the polynomial entries in the matrices $\Phi(x)$ and $\Psi(x)$, then Proposition 2.6 in \cite{PaperTricomiWeights} ensures the $2$-orthogonality of the polynomial sequence $\dis\polyseq[(n+1)^{-1}\,P'_{n+1}(x;a,b;c,d)]$ with respect to the vector of weights $\dis x\Phi(x)\Wvec(x)$ whilst  Proposition 2.7 in  \cite{PaperTricomiWeights} implies that, if $R_n(x)$ is the type I function for the index of length $n$ on the step line with respect to $\dis x\Phi(x)\Wvec(x)$, then $\dis -n^{-1}\,R'_n(x)$ is the type I function for the index of length $n+1$ on the step line with respect to the vector of weights $\dis\Wvec(x)$.

Therefore, by virtue of \eqref{shift on the parameters - hypergeometric weights}, we conclude that both \eqref{differentiation formula for type II polynomials - hypergeometric weights} and \eqref{differentiation formula for type I polynomials - hypergeometric weights} hold.
\end{proof}


\subsection{Type I multiple orthogonal polynomials}\label{Type I multiple orthogonal polynomials}
Due to the differential relation \eqref{differentiation formula for type I polynomials - hypergeometric weights}, the type I functions on the step line can be generated by concatenated differentiation of the weight function or, in other words, via a Rodrigues-type formula as it follows.
\begin{theorem}
\label{Rodrigues formula for the type I polynomials - theorem - hypergeometric weights}
For $a,b,c,d\in\R^+$ such that $\min\{c,d\}>\max\{a,b\}$ and $n\in\N$, let $Q_{n+1}(x;a,b;c,d)$ be the type I function for the index of length $n+1$ on the step line with respect to $\dis\Wvec(x;a,b;c,d)$.
Then
\begin{align}
\label{Rodrigues formula for the type I polynomials - hypergeometric weights}
\hspace{-.5cm}Q_{n+1}(x;a,b;c,d)
=\frac{(-1)^n}{n!}\DiffOpHigherOrder{n}
\left(\W\left(x;a+n,b+n;c+\floor{\frac{n+1}{2}}+n,d+\floor{\frac{n}{2}}+n\right)\right). 
\end{align}
\end{theorem}

\begin{proof}
We proceed by induction on $n\in\N$.

For $n=0$, \eqref{Rodrigues formula for the type I polynomials - hypergeometric weights} reads as $\dis Q_1(x;a,b;c,d)=\W(x;a,b;c,d)$, which trivially holds.
	
Using \eqref{differentiation formula for type I polynomials - hypergeometric weights} and then evoking the assumption that \eqref{Rodrigues formula for the type I polynomials - hypergeometric weights} holds for a fixed $n\in\N$, we obtain
\begin{align*}
& Q_{n+2}(x;a,b;c,d)  
=-\frac{1}{n+1}\DiffOp\left(Q_{n+1}(x;a+1,b+1;d+1,c+2)\right) \\
&\quad
=\frac{(-1)^{n+1}}{(n+1)!}\DiffOpHigherOrder{n+1} \left(\W\left(x;a+1+n,b+1+n;d+1+\floor{\frac{n+1}{2}}+n,c+2+\floor{\frac{n}{2}}+n\right)\right) \\
&\quad
=\frac{(-1)^{n+1}}{(n+1)!}\DiffOpHigherOrder{n+1} \left(\W\left(x;a+n+1,b+n+1;c+\floor{\frac{n+2}{2}}+n+1,d+\floor{\frac{n+1}{2}}+n+1\right)\right).
\end{align*}
If we equate the first and latter members, we obtain \eqref{Rodrigues formula for the type I polynomials - hypergeometric weights} for $n+1$ and the result follows by induction.
\end{proof}

We continue with further properties reagrding   type I polynomials $(A_n,B_n)$ in \eqref{typeI An Bn} associated with the type I function $Q_n(x)$ in \eqref{Rodrigues formula for the type I polynomials - hypergeometric weights}.  
In fact, Theorem \ref{canonical matrix differential equation satisfied by the hypergeometric weights - theorem} combined with the proof of \cite[Prop.~2.7]{PaperTricomiWeights} leads to the following differential-difference relation between the pair of polynomials 
\[\left(A_{n+1}(x),B_{n+1}(x)\right):=\left(A_{n+1}(x;a,b;c,d),B_{n+1}(x;a,b;c,d)\right)\]
and 
\[\left(C_n(x),D_n(x)\right):=\left(A_n(x;a+1,b+1;d+2,c+1),B_n(x;a+1,b+1;d+2,c+1)\right), \]
the polynomials on the step line satisfying multiple orthogonality relations of type I with respect to $\dis\Wvec(x;a,b;c,d)$ and to $\dis\Wvec(x;a+1,b+1;d+1,c+2)$: 
\begin{subequations}
\begin{equation}
\label{recursive relation for type I polynomials wrt to hypergeometric weights 1}
A_{n+1}(x)
=\frac{c(c+1)d}{nab}\left(-\frac{1}{c-b}\left(b\,C_n(x)+x\,C'_n(x)\right)
+\frac{(d+1)x}{(b+1)(d-a)}\left(d\,D_n(x)+x\,D'_n(x)\right)\right)
\end{equation}
and
\begin{equation}
\label{recursive relation for type I polynomials wrt to hypergeometric weights 2}
B_{n+1}(x)
=\frac{(c+1)d}{na}\left(\frac{1}{c-b}\left(c\,C_n(x)+\frac{x}{b}\,C'_n(x)\right)
-\frac{d+1}{(b+1)(d-a)}\left(a\,D_n(x)+c\,x\,D'_n(x)\right)\right)
\end{equation}
\end{subequations}
which hold for all $n\geq 1$. 

Formulas \eqref{recursive relation for type I polynomials wrt to hypergeometric weights 1}-\eqref{recursive relation for type I polynomials wrt to hypergeometric weights 2} can be used to recursively generate type I polynomials with respect to $\dis\Wvec(x;a,b;c,d)$. The latter can be written as follows
\[
\frac{n a}{c(c+1)d}
\left[ \begin {array}{cc}
	b &0\\
	0&c 
\end{array}\right] 
\twovector{ A_{n+1}(x)}{ B_{n+1}(x)}
=  
\left[ \begin {array}{cc}
	b & -d x\\
	-c & a 
\end{array}\right] M 
\twovector{ C_{n}(x)}{ D_{n}(x)}
+ \ 
x \,\left[ \begin {array}{cc}
	1 & -x\\
	-1/b & c 
\end{array}\right] M
\twovector{ C_{n}'(x)}{ D_{n}'(x)}
\]
where
\[
M=\left[ \begin {array}{cc} 
	-{\frac {1}{c-b}}		& 0 
\\ 
	 0	&-{\frac { \left( d+1 \right) }{\left( b+1 \right)  \left( d-a \right) }}\end {array} \right] 
\]
or, equivalently, 
\[
\frac{n a b   }{(c+1)d}
\twovector{ A_{n+1}(x)}{ B_{n+1}(x)}
=   
\left[ \begin {array}{cc} bc&-cdx\\ \noalign{\medskip}-bc&ba
\end {array} \right] 
M 
\twovector{ C_{n}(x)}{ D_{n}(x)}
+ \ 
x \,
 \left[ \begin {array}{cc} c&-cx\\ \noalign{\medskip}-1&bc\end {array}
 \right] 
M
\twovector{ C_{n}'(x)}{ D_{n}'(x)}. 
\]
Thus, the type I polynomials can be generated by the rising operator 
\begin{align*}
\mathcal{O}(a,b;c,d)
=\frac{(c+1)d}{ab}\left(
\left[\begin {array}{cc} -\dfrac{bc}{c-b}&\dfrac{cdx}{c-b}\\ \noalign{\medskip}\dfrac{bc}{c-b}&-\dfrac{ab}{c-b}\end {array} \right] 
+\left[\begin{array}{cc}-\dfrac{c(d+1)}{(b+1)(d-a)}c&\dfrac{c(d+1)x}{(b+1)(d-a)}\\\noalign{\medskip}\dfrac{(d+1)}{(b+1)(d-a)}&\dfrac{-bc(d+1)}{(b+1)(d-a)}\end{array}\right]\,x\,\partial_x\right),
\end{align*}
since we have 
\begin{align*}
\twovector{ A_{n+1}(x;a,b;c,d)}{B_{n+1}(x;a,b;c,d)}
=\frac{1}{n}\,\mathcal{O}\twovector{ A_n(x;a+1,b+1;d+2,c+1)}{B_n(x;a+1,b+1;d+2,c+1)}. 
\end{align*}
As a result, we obtain the matrix Rodrigues-type formula for type I polynomials
\begin{align*}
\twovector{A_{n+1}(x;a,b;c,d)}{B_{n+1}(x;a,b;c,d)}
=\frac{1}{n!}\left(\prod_{k=0}^{n-1}\mathcal{O}(a_k,b_k;c_k,d_k)\right)\twovector{1}{0},
\quad n\in\N,
\end{align*}
where, as usual, the product of differential operators is understood as the composition, and the parameters involved are as follows 
$a_k=a+k$, $b_k=b+k$, $c_{2j}=c+3j$ and $c_{2j+1}=d+3j+2$, $d_{2j}=d+3j$ and $d_{2j+1}=c+3j+1$.

\section{Characterisation of the type II polynomials}
\label{Type II MOPs wrt hypergeometric weights}

The type II multiple orthogonal polynomials on the step line are described in detail here.
This characterisation includes: their explicit expression in Theorem \ref{explicit formulas for the 2-orthogonal polynomials - hypergeometric weights}, a third order linear differential equation with polynomial coefficients in Theorem \ref{differential equation satisfied by 2-OPS - hypergeometric weights}, a third order recurrence in Theorem \ref{recurrence relation satisfied by the 2-OPS - 1st theorem}. 
The asymptotic properties of these polynomials are analysed in \S \ref{Jacobi-Pineiro polynomials and asymptotic results}, which coincide with those observed for Jacobi-Pi\~neiro polynomials. 
We give the ratio asymptotics of two consecutive polynomials, the limiting zero distribution as well as a Mehler-Heine formula for the behaviour near the endpoint $0$. At last, we analyse particular realisations of these polynomials. 
Namely, the connection to Jacobi-type $2$-orthogonal polynomials, in \S \ref{Jacobi-type 2-OPS and sequence with constant recurrence relation coefficients}, and the connection to the cubic components of Hahn-classical threefold symmetric polynomials in \S \ref{Hahn-classical 3-fold symmetric 2-OPS}, where we also describe confluence relations to another (Hahn-classical) polynomials that are $2$-orthogonal with respect to weights involving the  confluent hypergeometric functions of the second kind. 

\subsection{Explicit expression}
\label{explicit expression}
Based on the moments expression \eqref{moments of the hypergeometric weight}, we deduce an explicit representation for the type II multiple orthogonal polynomials on the step line with respect to $\dis\Wvec(x;a,b;c,d)$ as generalised hypergeometric series. 

\begin{theorem}
\label{explicit formulas for the 2-orthogonal polynomials - hypergeometric weights}
For $a,b,c,d\in\R^+$ such that $\min\{c,d\}>\max\{a,b\}$, let $\dis\polyseq[P_n(x):=P_n(x;a,b;c,d)]$ be the monic $2$-orthogonal polynomial sequence with respect to $\dis\Wvec(x;a,b;c,d)$. Then
\begin{align}
\label{explicit formula for the 2-orthogonal polynomials as a 3F2}
P_n(x)=
\frac{(-1)^n\pochhammer{a}\pochhammer{b}}{\pochhammer{c+\floor{\frac{n}{2}}}\pochhammer{d+\floor{\frac{n-1}{2}}}}
\,\threeFtwo{-n,c+\floor{\frac{n}{2}},d+\floor{\frac{n-1}{2}}}{a,b}.
\end{align}
\end{theorem}

By definition of the generalised hypergeometric series, the latter formula is equivalent to
\begin{align}
\label{expansion of the 2-OPS over the canonical basis - hypergeometric weights}
P_n(x)=\sum_{j=0}^{n}\tau_{n,j}x^{n-j},
\quad \text{with}\quad
\tau_{n,j}=
\binom{n}{j}\frac{(-1)^{j}\pochhammer[j]{a+n-j}\pochhammer[j]{b+n-j}}
{\pochhammer[j]{c+\floor{\frac{n}{2}}+n-j}\pochhammer[j]{d+\floor{\frac{n-1}{2}}+n-j}}.
\end{align}

To prove Theorem \ref{explicit formulas for the 2-orthogonal polynomials - hypergeometric weights} we need to show that the sequence $\polyseq$ defined by \eqref{explicit formula for the 2-orthogonal polynomials as a 3F2} satisfies the $2$-orthogonality conditions with respect to $\dis\Wvec(x;a,b;c,d)$, that is, we need to check that, for each $j\in\{0,1\}$,
\begin{align}
\label{orthogonality conditions hypergeometric weights}
\int_{0}^{1}x^kP_n(x)\W(x;a,b+j;c+j,d)\dx
=\begin{cases}
0, &\text{ if } n\geq 2k+j+1,\\
N_n(a,b;c,d)\neq 0, &\text{ if } n=2k+j.
\end{cases}
\end{align}
Actually, as we are dealing with a Nikishin system, the existence of a $2$-orthogonal polynomial sequence with respect to $\dis\Wvec(x;a,b;c,d)$ is guaranteed. 
By virtue of the generalised hypergeometric differential equation \eqref{generalised hypergeometric differential equation}, it is rather straightforward to show that the polynomials given by \eqref{explicit formula for the 2-orthogonal polynomials as a 3F2} satisfy the differential property \eqref{differentiation formula for type II polynomials - hypergeometric weights} stated in Theorem \ref{differentiation formulas for the multiple orthogonal polynomials - theorem - hypergeometric weights}. 
A property that a $2$-orthogonal polynomial sequence with respect to $\dis\Wvec(x;a,b;c,d)$ must satisfy. 
Therefore, it would be sufficient to check the orthogonality conditions \eqref{orthogonality conditions hypergeometric weights} when $k=0$ to then prove the result by induction on $n\in\N$ (the degree of the polynomials).
However, we opt for checking that the polynomials $P_n(x)$ in \eqref{explicit formula for the 2-orthogonal polynomials as a 3F2} satisfy all the orthogonality conditions \eqref{orthogonality conditions hypergeometric weights}. 
On the one hand, this process enables us to show directly that the polynomials in  \eqref{explicit formula for the 2-orthogonal polynomials as a 3F2} are indeed $2$-orthogonal with respect to $\dis\Wvec(x;a,b;c,d)$ without arguing with the Nikishin property. On the other hand, it provides a method to derive explicit expressions for the nonzero coefficients $N_n(a,b;c,d)$ in \eqref{orthogonality conditions hypergeometric weights} which are used in Subsection \ref{recurrence relation} to obtain explicit expressions for the nonzero $\gamma$-coefficients in the third order recurrence relation \eqref{recurrence relation for a 2-OPS} satisfied by these polynomials.

To compute the integrals in \eqref{orthogonality conditions hypergeometric weights}, we use the following auxiliary lemma, which can be found in \cite[Lemma 3.2]{PaperTricomiWeights}.
As mentioned therein, \eqref{Minton's formula} was deduced in \cite{Minton} and \eqref{hypergeometric formula 1} can be obtained by taking the limit $\beta\to +\infty$ in \eqref{Minton's formula}.

\begin{lemma}
\label{explicit formulas for hypergeometric series}
Let $n$, $p$, and $m_1,\cdots,m_p$ be positive integers such that $\dis m:=\sum_{i=1}^{p}m_i\leq n$ and $\beta,f_1,\cdots,f_p$ be complex numbers with positive real part.
Then
\begin{align}
\label{hypergeometric formula 1}
\Hypergeometric[1]{p+1}{p}{-n,f_1+m_1,\cdots,f_p+m_p}{f_1,\cdots,f_p}
=\begin{cases}
0 & \text{ if } m<n,\\
\dis\frac{(-1)^n n!}{\pochhammer[m_1]{f_1}\cdots\pochhammer[m_p]{f_p}} & \text{ if } m=n.
\end{cases}
\end{align}
and
\begin{align}
\label{Minton's formula}
\Hypergeometric[1]{p+2}{p+1}{-n,\beta,f_1+m_1,\cdots,f_p+m_p}{\beta+1,f_1,\cdots,f_p}=
\frac{n!\pochhammer[m_1]{f_1-\beta}\cdots\pochhammer[m_p]{f_p-\beta}}
{\pochhammer{\beta+1}\pochhammer[m_1]{f_1}\cdots\pochhammer[m_p]{f_p}}.
\end{align}
\end{lemma}

\begin{proof}[Proof of Theorem \ref{explicit formulas for the 2-orthogonal polynomials - hypergeometric weights}.] 
Recalling the explicit expression for $P_n(x)$ given by \eqref{explicit formula for the 2-orthogonal polynomials as a 3F2}, and the definition of the generalised hypergeometric series \eqref{generalised hypergeometric series}, we get, for both $j\in\{0,1\}$ and any $k,n\in\N$,
\begin{align*}
&\int_{0}^{1}x^kP_n(x)\W(x;a,b+j;c+j,d)\dx \\
=&\,\frac{(-1)^n\pochhammer{a}\pochhammer{b}}{\pochhammer{c+\floor{\frac{n}{2}}}\pochhammer{d+\floor{\frac{n-1}{2}}}} \,\sum_{i=0}^{n}\frac{\pochhammer[i]{-n}\pochhammer[i]{c+\floor{\frac{n}{2}}}\pochhammer[i]{d+\floor{\frac{n-1}{2}}}} {i!\pochhammer[i]{a}\pochhammer[i]{b}}\int_{0}^{1}x^{k+i}\W(x;a,b+j;c+j,d)\dx.
\end{align*}
Moreover, using the formula for the moments of the hypergeometric weight \eqref{moments of the hypergeometric weight},
\begin{align*}
&\int_{0}^{1}x^kP_n(x)\W(x;a,b+j;c+j,d)\dx \\
=&\,\frac{(-1)^n\pochhammer{a}\pochhammer{b}}{\pochhammer{c+\floor{\frac{n}{2}}}\pochhammer{d+\floor{\frac{n-1}{2}}}} \,\sum_{i=0}^{n}\frac{\pochhammer[i]{-n}\pochhammer[i]{c+\floor{\frac{n}{2}}}\pochhammer[i]{d+\floor{\frac{n-1}{2}}}} {i!\pochhammer[i]{a}\pochhammer[i]{b}} \frac{\pochhammer[k+i]{a}\pochhammer[k+i]{b+j}}{\pochhammer[k+i]{c+j}\pochhammer[k+i]{d}}.
\end{align*}
Therefore, recalling again the definition of the generalised hypergeometric series \eqref{generalised hypergeometric series}, we have
\begin{align}
\label{orthogonality conditions proof first integral}
\begin{split}
&\int_{0}^{1}x^kP_n(x)\W(x;a,b+j;c+j,d)\dx \\
=&\frac{(-1)^n\pochhammer{a}\pochhammer{b}\pochhammer[k]{a}\pochhammer[k]{b+j}}
{\pochhammer{c+\floor{\frac{n}{2}}}\pochhammer{d+\floor{\frac{n-1}{2}}}\pochhammer[k]{c+j}\pochhammer[k]{d}}
\,\Hypergeometric[1]{5}{4}{-n,a+k,b+k+j,c+\floor{\frac{n}{2}},d+\floor{\frac{n-1}{2}}}{a,b,c+k+j,d+k}.
\end{split}
\end{align}
For any $n\in\N$, $\dis\floor{\frac{n}{2}}+\floor{\frac{n-1}{2}}=n-1$ and if $n\geq 2k+j+1$, $j\in\{0,1\}$, then $\dis\floor{\frac{n}{2}}\geq k+j$ and $\dis\floor{\frac{n-1}{2}}\geq k$.
Therefore, using \eqref{hypergeometric formula 1} in Lemma \ref{explicit formulas for hypergeometric series}, we deduce that, for both $j\in\{0,1\}$,
\begin{align*}
\Hypergeometric[1]{5}{4}{-n,a+k,b+k+j,c+\floor{\frac{n}{2}},d+\floor{\frac{n-1}{2}}}{a,b,c+k+j,d+k}=0,
\quad \text{for any }\quad n\geq 2k+j+1,
\end{align*}
and, as a result,
\begin{equation}
\int_{0}^{1}x^kP_n(x)\W(x;a,b+j;c+j,d)\dx=0,
\quad \text{for any }\quad n\geq 2k+j+1.
\end{equation}
Taking $j=0$ and $n=2k$ in \eqref{orthogonality conditions proof first integral}, 
\begin{align*}
\int_{0}^{1}x^kP_{2k}(x)\W(x;a,b;c,d)\dx
=\frac{\pochhammer[2k]{a}\pochhammer[2k]{b}\pochhammer[k]{a}\pochhammer[k]{b}}
{\pochhammer[3k]{c}\pochhammer[3k-1]{d}(d+k-1)}
\,\Hypergeometric[1]{4}{3}{-2k,a+k,b+k,d+k-1}{a,b,d+k}.
\end{align*}
and, using \eqref{Minton's formula}, we get 
\begin{align*}
\Hypergeometric[1]{4}{3}{-2k,a+k,b+k,d+k-1}{a,b,d+k} 
=\frac{(2k)!\pochhammer[k]{a-d+1-k}\pochhammer[k]{b-d+1-k}}
{\pochhammer[2k]{d+k}\pochhammer[k]{a}\pochhammer[k]{b}}
=\frac{(2k)!\pochhammer[k]{d-a}\pochhammer[k]{d-b}}
{\pochhammer[k]{a}\pochhammer[k]{b}\pochhammer[2k]{d+k}}.
\end{align*}
Therefore,
\begin{align}
\label{nonzero integral first weight}
\int_{0}^{1}x^kP_{2k}(x)\W(x;a,b;c,d)\dx
=\frac{(2k)!\pochhammer[2k]{a}\pochhammer[2k]{b}\pochhammer[k]{d-a}\pochhammer[k]{d-b}}
{\pochhammer[3k]{c}\pochhammer[3k]{d}\pochhammer[2k]{d+k-1}}>0,
\end{align}
and \eqref{orthogonality conditions hypergeometric weights} holds for any $k,n\in\N$ when $j=0$.

Similarly, taking $j=1$ and $n=2k+1$ in \eqref{orthogonality conditions proof first integral}, 
\begin{align*}
\int_{0}^{1}x^kP_{2k+1}(x)\W(x;a,b+1;c+1,d)\dx
=-\frac{\pochhammer[2k+1]{a}\pochhammer[2k+1]{b}\pochhammer[k]{a}\pochhammer[k]{b+1}}
{\pochhammer[3k]{c+1}(c+k)\pochhammer[3k+1]{d}}
\,\Hypergeometric[1]{4}{3}{-2k-1,a+k,b+k+1,c+k}{a,b,c+k+1}.
\end{align*}
and, using again \eqref{Minton's formula}, we get 
\begin{align*}
\Hypergeometric[1]{4}{3}{-2k-1,a+k,b+k+1,c+k}{a,b,c+k+1}
=-\frac{(2k+1)!\pochhammer[k]{c-a+1}\pochhammer[k+1]{c-b}}
{\pochhammer[k]{a}\pochhammer[k+1]{b}\pochhammer[2k+1]{c+k+1}},
\end{align*}
so that
\begin{align}
\label{nonzero integral second weight}
\int_{0}^{1}x^kP_{2k+1}(x)\W(x;a,b+1;c+1,d)\dx
=\frac{(2k+1)!\pochhammer[2k+1]{a}\pochhammer[2k]{b+1}\pochhammer[k]{c-a+1}\pochhammer[k+1]{c-b}}
{\pochhammer[3k+1]{c+1}\pochhammer[2k+1]{c+k}\pochhammer[3k+1]{d}}>0.
\end{align}
ensuring that \eqref{orthogonality conditions hypergeometric weights} also holds for any $k,n\in\N$ when $j=1$.
\end{proof}


\subsection{Differential equation}
\label{differential equation}
The type II multiple orthogonal polynomials of hypergeometric type described in \eqref{explicit formula for the 2-orthogonal polynomials as a 3F2} are solutions to the following third order differential equation.
\begin{theorem}
\label{differential equation satisfied by 2-OPS - hypergeometric weights}
For $a,b,c,d\in\R^+$ such that $\min\{c,d\}>\max\{a,b\}$, let $\dis\polyseq[P_n(x):=P_n(x;a,b;c,d)]$ be the monic $2$-orthogonal polynomial sequence with respect to $\dis\Wvec(x;a,b;c,d)$. 
Then
\begin{align}
\label{differential equation satisfied by the 2-orthogonal polynomials - hypergeometric weights}
x^2(1-x)P'''_n(x)-x\varphi(x)P''_n(x)+\psi_n(x)P'_n(x)+n\lambda_nP_n(x)=0, 
\end{align}
with 
\begin{align*}
&\varphi(x)=(c+d+2)x-(a+b+1), 
\\
&\psi_n(x)=\big((n-1)(c+d+n)-\lambda_n\big)x+ab,
\\
&\lambda_n=\left(c+\floor{\frac{n}{2}}\right)\left(d+\floor{\frac{n-1}{2}}\right).
\end{align*}
\end{theorem}
\begin{proof}
Combining the explicit formula for the 2-orthogonal polynomials as terminating hypergeometric series \eqref{explicit formula for the 2-orthogonal polynomials as a 3F2} and the generalised hypergeometric differential equation \eqref{generalised hypergeometric differential equation}, we obtain
\begin{align}
\begin{split}
\label{diff eq first version - hypergeometric weights}
&\,\left[\left(x\,\DiffOp+a\right)\left(x\,\DiffOp+b\right)\DiffOp\right]P_n(x) \\
=&\,\left[\left(x\,\DiffOp-n\right)\left(x\,\DiffOp+c+\floor{\frac{n}{2}}\right) \left(x\,\DiffOp+d+\floor{\frac{n-1}{2}}\right)\right]P_n(x).
\end{split}
\end{align}
Expanding the left-hand side of \eqref{diff eq first version - hypergeometric weights}, we get
\begin{align*}
\left[\left(x\,\DiffOp+a\right)\left(x\,\DiffOp+b\right)\DiffOp\right]P_n(x)=x^2\,P'''_n(x)+(a+b+1)x\,P''_n(x)+abP'_n(x).
\end{align*}
Similarly, recalling that $\dis\floor{\frac{n}{2}}+\floor{\frac{n-1}{2}}=n-1$, for any $n\in\N$, we derive
\begin{align*}
\left[\left(x\,\DiffOp+c+\floor{\frac{n}{2}}\right)\left(x\,\DiffOp+d+\floor{\frac{n-1}{2}}\right)\right]P_n(x)
=x^2P''_n(x)+(c+d+n)x\,P'_n(x)+\epsilon_nP_n(x).
\end{align*}
Therefore, the right-hand side of \eqref{diff eq first version - hypergeometric weights} is
\begin{align*}
&\left[\left(x\,\DiffOp-n\right)\left(x\,\DiffOp+c+\floor{\frac{n}{2}}\right) \left(x\,\DiffOp+d+\floor{\frac{n-1}{2}}\right)\right]P_n(x)\\
=\,&x^3\,P'''_n(x)+(c+d+2)x^2\,P''_n(x)+\left(\lambda_n-(n-1)(c+d+n)\right)xP'_n(x)-n\lambda_nP_n(x).
\end{align*}
Finally, combining the expressions for both sides of \eqref{diff eq first version - hypergeometric weights}, we derive the differential equation \eqref{differential equation satisfied by the 2-orthogonal polynomials - hypergeometric weights}.
\end{proof}


\subsection{Recurrence relation}
\label{recurrence relation}
As a $2$-orthogonal sequence, the hypergeometric type polynomials expressed by \eqref{explicit formula for the 2-orthogonal polynomials as a 3F2} satisfy a third order recurrence relation of the form
\begin{align}
\label{recurrence relation satisfied by a 2-OPS}
P_{n+1}(x)=\left(x-\beta_n\right)P_n(x)-\alpha_nP_{n-1}(x)-\gamma_{n-1}P_{n-2}(x),
\end{align}
Our purpose here is to obtain explicit expressions for the recurrence coefficients involved.
The linear independence of $\{x^n\}_{n\in\N}$ implies that we can equate their coefficients on both sides of the recurrence relation \eqref{recurrence relation satisfied by a 2-OPS}.
After equating the coefficients of $x^n$ and $x^{n-1}$ we obtain, respectively,
\begin{equation*}
\beta_n=\tau_{n,1}-\tau_{n+1,1}
\quad\text{and}\quad
\alpha_n=\tau_{n,2}-\tau_{n+1,2}-\left(\tau_{n,1}\right)^2+\tau_{n,1}\tau_{n+1,1},
\end{equation*}
where, based on \eqref{expansion of the 2-OPS over the canonical basis - hypergeometric weights}, we have 
\begin{equation*}
\tau_{n,1}=-\frac{n(a+n-1)(b+n-1)}{\left(c+\floor{\frac{n}{2}}+n-1\right)\left(d+\floor{\frac{n-1}{2}}+n-1\right)}
\end{equation*}
and
\begin{equation*}
\tau_{n,2}=\frac{n(a+n-1)(b+n-1)(n-1)(a+n-2)(b+n-2)}
{2\left(c+\floor{\frac{n}{2}}+n-1\right)\left(d+\floor{\frac{n-1}{2}}+n-1\right) \left(c+\floor{\frac{n}{2}}+n-2\right)\left(d+\floor{\frac{n-1}{2}}+n-2\right)}.
\end{equation*}

Hence we derive that, for each $k\in\N$,
\begin{align*}
\beta_{2k}(a,b;c,d)=
\frac{(2k+1)(a+2k)(b+2k)}{(c+3k)(d+3k)}-\frac{2k(a+2k-1)(b+2k-1)}{(c+3k-1)(d+3k-2)}
\end{align*}
and
\begin{align*}
\beta_{2k+1}(a,b;c,d)=
\frac{(2k+2)(a+2k+1)(b+2k+1)}{(c+3k+2)(d+3k+1)}-\frac{(2k+1)(a+2k)(b+2k)}{(c+3k)(d+3k)}
\end{align*}
as well as  
\begin{align*}
\begin{split}
\alpha_{2k+1}(a,b;c,d) &
=\frac{(2k+1)(a+2k)(b+2k)}{(c+3k)(d+3k)}
\Bigg(\frac{k(a+2k-1)(b+2k-1)}{(c+3k-1)(d+3k-1)} \\&
-\frac{(2k+1)(a+2k)(b+2k)}{(c+3k)(d+3k)}
+\frac{(k+1)(a+2k+1)(b+2k+1)}{(c+3k+1)(d+3k+1)}\Bigg),
\end{split}
\end{align*}
and
\begin{align*}
\begin{split}
\alpha_{2k+2}(a,b;c,d) &
=\frac{2(k+1)(a+2k+1)(b+2k+1)}{(c+3k+2)(d+3k+1)}
\Bigg(\frac{(2k+1)(a+2k)(b+2k)}{2(c+3k+1)(d+3k)} \\&
-\frac{2(k+1)(a+2k+1)(b+2k+1)}{(c+3k+2)(d+3k+1)}
+\frac{(2k+3)(a+2k+2)(b+2k+2)}{2(c+3k+3)(d+3k+2)}\Bigg).
\end{split}
\end{align*}
The expressions for the coefficients $\gamma_n$ in \eqref{recurrence relation satisfied by a 2-OPS} could also be obtained in an analogous way after comparing the coefficients of $x^{n-2}$. 
However, it is easier to derive such expressions directly from the $2$-orthogonality conditions, which, applied to the recurrence relation \eqref{recurrence relation satisfied by a 2-OPS}, imply that, for each $k\in\N$ and $j\in\{0,1\}$,
\begin{align*}
\gamma_{2k+1+j}(a,b;c,d)
=\frac{\dis\int_{0}^{1}x^{k+1}P_{2k+2+j}(x;a,b;c,d)\W(x;a,b+j;c+j,d)\dx}
{\dis\int_{0}^{1}x^kP_{2k+j}(x;a,b;c,d)\W(x;a,b+j;c+j,d)\dx}.
\end{align*}
Based on the latter alongside with \eqref{nonzero integral first weight} and \eqref{nonzero integral second weight}, we deduce that, for all $k\in\N$,
\begin{align*}
\gamma_{2k+1}(a,b;c,d)
=\frac{\pochhammer[2]{2k+1}\pochhammer[2]{a+2k}\pochhammer[2]{b+2k}(d-1+k)(d-a+k)(d-b+k)}
{\pochhammer[3]{c+3k}\pochhammer[3]{d+3k-1}\pochhammer[3]{d+3k}}.
\end{align*}
and
\begin{align*}
\gamma_{2k+2}(a,b;c,d)
=\frac{\pochhammer[2]{2k+2}\pochhammer[2]{a+2k+1}\pochhammer[2]{b+2k+1}(c+k)(c-a+k+1)(c-b+k+1)}
{\pochhammer[3]{c+3k+1}\pochhammer[3]{c+3k+2}\pochhammer[3]{d+3k+1}}
\end{align*}

As a consequence, we have just proved the following result. 

\begin{theorem}
\label{recurrence relation satisfied by the 2-OPS - 1st theorem}
For $a,b,c,d\in\R^+$ satisfying \eqref{parameters}, let $\dis\polyseq[P_n(x):=P_n(x;a,b;c,d)]$ be the monic $2$-orthogonal polynomial sequence with respect to $\dis\Wvec(x;a,b;c,d)$.
Then $\dis\polyseq$ satisfies the recurrence relation
\begin{align}
\label{recurrence relation satisfied by the 2-OPS - hypergeometric weights}
P_{n+1}(x)=\left(x-\beta_n\right)P_n(x)-\alpha_nP_{n-1}(x)-\gamma_{n-1}P_{n-2}(x),
\end{align}
where, for each $n\in\N$,
\begin{subequations}
\label{recurrence relation coefficients - hypergeometric weights}
\begin{align}
\label{betas - hypergeometric weights}
\beta_n
=\frac{(n+1)(a+n)(b+n)}{\left(c'_{n-1}+n\right)\left(c'_n+n\right)}
-\frac{n(a+n-1)(b+n-1)}{\left(c'_{n-1}+n-1\right)\left(c'_n+n-2\right)},
\end{align}
\begin{align}
\begin{split}
\label{alphas - hypergeometric weights}
\alpha_{n+1} &
=\frac{(n+1)(a+n)(b+n)}{(c'_{n-1}+n)(c'_n+n)}
\Bigg(\frac{n(a+n-1)(b+n-1)}{2(c'_{n-1}+n-1)(c'_n+n-1)} \\&
-\frac{(n+1)(a+n)(b+n)}{(c'_{n-1}+n)(c'_n+n)}
+\frac{(n+2)(a+n+1)(b+n+1)}{2(c'_{n-1}+n+1)(c'_n+n+1)}\Bigg). 
\end{split}
\end{align}
and
\begin{align}
\label{gammas - hypergeometric weights}
\gamma_{n+1}
=\frac{\pochhammer[2]{n+1}\pochhammer[2]{a+n}\pochhammer[2]{b+n}(c'_n-1)(c'_n-a)(c'_n-b)}
{\pochhammer[3]{c'_{n-1}+n}\pochhammer[3]{c'_n+n}\pochhammer[3]{c'_n+n-1}},
\end{align}
\end{subequations}
with
\begin{align}
\label{c'_n}
c'_n=
\begin{cases}
c+k & \text{if }n=2k-1, \\
d+k & \text{if }n=2k.
\end{cases}
\end{align}

\end{theorem}

With the purpose of rewriting the recurrence relation coefficients using more convenient expressions, we introduce a set of positive coefficients $\dis\left(\lambda_k=\lambda_k\left(a,b;c,d\right)\right)_{k\in\N}$, involving the $c'_n$ introduced in the latter theorem and defined by
\begin{align}
\label{coefficients of the branched continued fraction - hypergeometric weights}
\begin{cases}
\dis\lambda_{3n}=\frac{n(b+n-1)(c'_{n}-a-1)}{(c'_{n}+n-2)(c'_{n}+n-1)(c'_{n-1}+n-1)},
\vspace*{0,25 cm}\\
\dis\lambda_{3n+1}=\frac{n(a+n)(c'_{n-1}-b)}{(c'_{n}+n-1)(c'_{n-1}+n-1)(c'_{n-1}+n)},
\vspace*{0,25 cm}\\
\dis\lambda_{3n+2}=\frac{(a+n)(b+n)(c'_{n}-1)}{(c'_{n}+n-1)(c'_{n}+n)(c'_{n-1}+n)}.
\end{cases}
\end{align}

The coefficients above were obtained from \cite[Th.~14.5]{AlanSokalEtAlBranchedContinuedFractions} as the coefficients of a branched continued fraction representation for $\HypergeometricOneLine[t]{3}{2}{a,b,1}{c,d}$, the ordinary generating function of the moment sequence given by \eqref{moments of the hypergeometric weight}.

Observe that $\lambda_0=\lambda_1=0$ and $\lambda_k>0$, for all $k\geq 2$.
In addition, $\dis\lambda_k\to\frac{4}{27}$, as $k\to\infty$, and we have, for all $n\in\N$,
\begin{itemize}
\item $\dis\beta_n=\lambda_{3n}+\lambda_{3n+1}+\lambda_{3n+2}$; \vspace*{0,1 cm}
\hfill\refstepcounter{equation}\textup{(\theequation)}\label{betas}
 
\item $\dis\alpha_{n+1}=\lambda_{3n+1}\lambda_{3n+3}+\lambda_{3n+2}\lambda_{3n+3}+\lambda_{3n+2}\lambda_{3n+4}$;
\hfill\refstepcounter{equation}\textup{(\theequation)}\label{alphas}

\item $\dis\gamma_{n+1}=\lambda_{3n+2}\lambda_{3n+4}\lambda_{3n+6}$.
\hfill\refstepcounter{equation}\textup{(\theequation)}\label{gammas}
\end{itemize}

Therefore, we can rewrite Theorem \ref{recurrence relation satisfied by the 2-OPS - 1st theorem} as the following result.
\begin{theorem}
\label{recurrence relation satisfied by the 2-OPS and asymptotic behaviour of the recurrence coefficients revisited}
For $a,b,c,d\in\R^+$ satisfying \eqref{parameters}, let $\dis\polyseq[P_n(x):=P_n(x;a,b;c,d)]$ be the monic $2$-orthogonal polynomial sequence with respect to $\dis\Wvec(x;a,b;c,d)$ and let the coefficients $\lambda_k$, $k\in\N$, be defined by \eqref{coefficients of the branched continued fraction - hypergeometric weights}
Then $\dis\polyseq$ satisfies the recurrence relation \eqref{recurrence relation satisfied by the 2-OPS - hypergeometric weights}, with coefficients given by \eqref{betas}-\eqref{gammas}. 
Therefore, the recurrence coefficients are real, positive and bounded with asymptotic behaviour
\begin{align}
\label{limits of the recurrence relation coefficients}
\beta_n\to 3\left(\frac{4}{27}\right)=\frac{4}{9},
\quad
\alpha_n\to 3\left(\frac{4}{27}\right)^2=\frac{16}{243}
\quad\text{and}\quad
\gamma_n\to \left(\frac{4}{27}\right)^3=\frac{64}{19683},
\quad\text{as} \ n\to\infty.
\end{align}
\end{theorem}

The expressions \eqref{betas}-\eqref{gammas} for the recurrence coefficients lead to a decomposition of the lower-Hessenberg matrix \eqref{Hessenberg matrix} as a product of three bidiagonal matrices with positive entries in the nonzero diagonals.
Thus, \eqref{Hessenberg matrix} is a special type of totally positive matrix, an oscillatory matrix (see \cite{GantmacherKreinOscillationMatrices}). As a result, we can conclude that the zeros of $P_n(x)$, which correspond to the eigenvalues of $H_n$, are real and positive as well as that the zeros of consecutive polynomials interlace, similarly to the main result of \cite[\S 9.2]{SaibNewPerspectivesOnD-OPS}.
Furthermore, applying \cite[Th.~3.5]{PaperTricomiWeights} to this case, with $\beta=\frac{4}{9}$, $\alpha=\frac{16}{243}$ and $\gamma=\frac{64}{19683}$, we guarantee that the zeros have absolute value less than $1$.
Therefore, we have an alternative proof, independent of the system being Nikishin, that the zeros of $\dis\polyseq[P_n(x;a,b;c,d)]$ are all located in the interval $(0,1)$ and that the zeros of consecutive polynomials interlace.


\subsection{Asymptotic behaviour. Connection with Jacobi-Pi\~neiro polynomials}
\label{Jacobi-Pineiro polynomials and asymptotic results}

Jacobi-Pi\~neiro polynomials are multiple orthogonal polynomials with respect to several classical Jacobi weights on the same interval.
They are usually defined as the multiple orthogonal polynomials with respect to measures $\dis\left(\mu_0,\cdots,\mu_{r-1}\right)$ supported on the interval $(0,1)$, with $\dis\dx[\mu_i](x)=x^{\alpha_i}(1-x)^\beta\dx$ for some $\beta,\alpha_1,\cdots,\alpha_r>-1$ such that $\alpha_i-\alpha_j\not\in\Z$ for any $i\neq j$.
These polynomials were introduced by Pi\~neiro in \cite{Pineiro} with $\beta=0$.
See \cite{WalterCoussementSomeClassicalMOPs} for a Rodrigues formula generating the type II Jacobi-Pi\~neiro polynomials as well as explicit expressions for the polynomials and for their recurrence relation coefficients.

The asymptotic behaviour of the recurrence relation coefficients in Theorem \ref{recurrence relation satisfied by the 2-OPS - 1st theorem} coincides with the asymptotic behaviour obtained in \cite{WalterCoussementSomeClassicalMOPs} for the coefficients of the recurrence relation satisfied by the Jacobi-Pi\~neiro polynomials.
Based on this relation, we show in this subsection that the polynomials investigated here share the ratio asymptotics, the asymptotic zero distribution and a Mehler-Heine asymptotic formula near the endpoint $0$ with the Jacobi-Pi\~neiro polynomials.
In fact, the Jacobi-Pi\~neiro polynomials originally studied by Pi\~neiro in \cite{Pineiro} are a limiting case of the polynomials investigated here. Precisely, the choice of $c=a$ and $d=b+1$ gives  $\W(x;a,b;c,d)=bx^{b-1}$ and $\W(x;a,b+1;c+1,d)=ax^{a-1}$, and, for this reason, the explicit formulas for the polynomials obtained in \S\ref{explicit expression} and in \cite{Pineiro} coincide.

Due to the asymptotic behaviour of the recurrence  coefficients obtained in Theorem \ref{recurrence relation satisfied by the 2-OPS - 1st theorem} and to the zeros of $P_n(x)$ being real, simple and interlacing with the zeros of $P_{n+1}(x)$, for each $n\in\N$, as previously shown, we can use \cite[Lemma~3.2]{AptKalLagoRochaLimitBehaviour} and \cite[Th.~3.1]{WalterCoussementX2AsymptoticZeroDistribution} to derive that
\begin{align}
\label{ratio asymptotics for the 2-OPS wrt the hypergeometric weights}
\lim_{n\to\infty}\frac{P_n(x)}{P_{n+1}(x)}
=\rho(x):=\frac{27}{4}\left(\frac{3}{2}\,x^\frac{1}{3}
\left(\e^{\frac{4\pi i}{3}}\left(-1+\sqrt{1-x}\right)^{\frac{1}{3}}
+\e^{\frac{2\pi i}{3}}\left(-1-\sqrt{1-x}\right)^{\frac{1}{3}}\right)-1\right), 
\end{align}
uniformly on compact subsets of $\mathbb{C}\backslash{[0,1]}$. 
As explained in \cite{WalterCoussementX2AsymptoticZeroDistribution}, the knowledge of this ratio asymptotic leads to prove that 
\[
	\lim_{n\to \infty} \frac{P_n'(x)}{P_n(x)} = - \frac{\rho'(x)}{\rho(x)} , \ \text{for } \ x\in\mathbb{C}\backslash{[0,1]},
\]
which results in showing that there exists a limit for the 
normalised zero counting measure of $P_n(x)$ 
\[
	\nu_n:= \frac{1}{n} \sum_{P_n(x)=0} \delta_x, 
\]
as $n\to\infty$, in the sense of the weak limit of measures, ie
$\displaystyle\lim_{n\to\infty}\int f\mathrm{d}\nu_n=\int f\mathrm{d}\nu$ for every bounded and continuous function $f$ on $[0,1]$. Here  $\delta_x$ is the Dirac point mass at $x$. 
As such, it was proved that 
\[
	\lim_{n\to \infty}  \int \frac{1}{x-t} d\nu_n(t) 
	 = - \frac{\rho'(x)}{\rho(x)} , \ \text{for } \ x\in\mathbb{C}\backslash{[0,1]},
\]
and, as shown in \cite[Th.~2.1]{WalterCoussementX2AsymptoticZeroDistribution}  the limiting measure $\nu$ has density 
\begin{align}
\label{asymptotic zero distribution for the 2-OPS wrt the hypergeometric weights}
\DiffOpFunction{\nu}=
\begin{cases}
\dis\frac{\sqrt{3}}{4\pi}\, \frac{\left(1+\sqrt{1-x}\right)^{\frac{1}{3}}+\left(1-\sqrt{1-x}\right)^{\frac{1}{3}}}{x^{\frac{2}{3}}\sqrt{1-x}} & \text{if } x\in(0,1), \\[0.3cm]
0 & \text{elsewhere},
\end{cases}
\end{align}
which is the asymptotic zero distribution of $\polyseq$. 

Jacobi-Pi\~neiro polynomials (orthogonal with respect to two measures) on the step line and the polynomial sequence $\polyseq$ under analysis share the same  ratio asymptotics and the asymptotic zero distribution because their recurrence coefficients have the same asymptotic behaviour and their zeros are simple, real, satisfy the interlacing property and are located on the interval $[0,1]$. 


We also derive a Mehler-Heine asymptotic formula satisfied by the $2$-orthogonal polynomials $P_n(x;a,b;c,d)$ near the origin, which give us more information about the zeros near the endpoint $0$ of the orthogonality interval.
For that purpose, we recall that the generalised hypergeometric series ${}_{p+1}F_q$ satisfies the confluent relation (see \cite[Eq.~16.8.10]{DLMF})
\begin{align}
\label{confluent relation for the generalised hypergeometric series}
\lim_{|\alpha|\to\infty}
\Hypergeometric[\frac{z}{\alpha}]{p+1}{q}{\alpha_1,\cdots,\alpha_p,\alpha}{\beta_1,\cdots,\beta_q}
=\Hypergeometric[z]{p}{q}{\alpha_1,\cdots,\alpha_p}{\beta_1,\cdots,\beta_q},
\end{align}
whenever both sides of this relation are convergent.
Moreover, we recall Theorem \ref{explicit formulas for the 2-orthogonal polynomials - hypergeometric weights} to write 
\begin{align*}
\frac{(-1)^n\pochhammer{c+\floor{\frac{n}{2}}}\pochhammer{d+\floor{\frac{n-1}{2}}}}{\pochhammer{a}\pochhammer{b}}
\,P_n\left(\frac{z}{n^3};a,b;c,d\right)
=\threeFtwo[\frac{z}{n^3}]{-n,c+\floor{\frac{n}{2}},d+\floor{\frac{n-1}{2}}}{a,b}. 
\end{align*}
Clearly, $\dis c+\floor{\frac{n}{2}},d+\floor{\frac{n-1}{2}}\sim\frac{n}{2}$, as $n\to\infty$. 
So we apply \eqref{confluent relation for the generalised hypergeometric series} three consecutive times to the generalised hypergeometric series on the right-hand side of the latter equation to deduce a Mehler-Heine type formula near the endpoint $0$ 
\begin{align}
\label{Mehler-Heine formula 2-OPS hypergeometric weights}
&\lim_{n\to\infty}
\frac{(-1)^n\pochhammer{c+\floor{\frac{n}{2}}}\pochhammer{d+\floor{\frac{n-1}{2}}}}{\pochhammer{a}\pochhammer{b}}
\,P_n\left(\frac{z}{n^3};a,b;c,d\right)
=\Hypergeometric[-\frac{z}{4}]{0}{2}{-}{a,b},
\end{align}
which converges uniformly on compact subsets of $\mathbb{C}$.

Note that the limit in this Mehler-Heine formula coincides with the limit in the Mehler-Heine formula for the Jacobi-Pi\~neiro polynomials obtained in \cite[Th.~2]{WalterMehlerHeineAsymptoticsForMOPs}, with $r=2$ and $\dis q_1=q_2={1}/{2}$.

Furthermore, we can derive a result about the asymptotic behaviour of the $k$-th smallest zero  of $P_n(x;a,b;c,d)$, which also coincides with the one obtained in \cite[\S 4]{WalterMehlerHeineAsymptoticsForMOPs} for the zeros of the $2$-orthogonal Jacobi-Pi\~neiro polynomials.
In fact, if we denote the zeros of $P_n(x;a,b;c,d)$ by $\left(x_k^{(n)}\right)_{1\leq k\leq n}$ and the zeros of the generalised hypergeometric series $\dis\HypergeometricOneLine[-z]{0}{2}{-}{a,b}$, which are all real and positive, by $\left(f_k\right)_{k\in\Z^+}$, with the zeros written in increasing order for both cases, then we have
\begin{align*}
\lim_{n\to\infty}n^3x_k^{(n)}=4f_k.
\end{align*}

\subsection{Particular cases: Jacobi-type 2-orthogonal polynomials and a sequence with constant recurrence relation coefficients}
\label{Jacobi-type 2-OPS and sequence with constant recurrence relation coefficients}

Using the coefficients $c'_n$ introduced in \eqref{c'_n}, the explicit expression for the type II polynomials given by \eqref{explicit formula for the 2-orthogonal polynomials as a 3F2} can be rewritten as
\begin{align}
\label{explicit formula for the 2-orthogonal polynomials as a 3F2 rewritten}
P_n(x;a,b;c,d)=
\frac{(-1)^n\pochhammer{a}\pochhammer{b}}{\pochhammer{c'_{n-2}}\pochhammer{c'_{n-1}}}
\,\threeFtwo{-n,c'_{n-2},c'_{n-1}}{a,b}
\end{align}
Furthermore, if $\dis d=c+\frac{1}{2}$, then $c'_n=\dis c+\frac{n+1}{2}$, for any $n\in\N$, and the expression above becomes
\begin{align}
\label{explicit formula for the 2-orthogonal polynomials as a 3F2 when d=c+1/2}
P_n\left(x;a,b;c,c+\frac{1}{2}\right)=
\frac{(-4)^n\pochhammer{a}\pochhammer{b}}{\pochhammer[2n]{2c-1+n}}\,\threeFtwo{-n,c+\frac{n-1}{2},c+\frac{n}{2}}{a,b}.
\end{align}
The latter polynomials coincide, up to a linear transformation of the variable, with the Jacobi-type $2$-orthogonal polynomials investigated in \cite{LamiriOuni}, with $\dis c=\frac{\nu+1}{2}$.

A particular case of \eqref{explicit formula for the 2-orthogonal polynomials as a 3F2 when d=c+1/2} which is worth of interest arises if we set $\dis(a,b;c,d)=\left(\frac{4}{3},\frac{5}{3};2,\frac{5}{2}\right)$. 
This choice of parameters gives  
\begin{align}
\label{explicit expression for the 2-OPS with constant recurrence coefficients}
P_n\left(x;\frac{4}{3},\frac{5}{3};2,\frac{5}{2}\right)
=\frac{(n+1)(n+2)}{2}\left(\frac{-4}{27}\right)^n
\threeFtwo{-n,\frac{n+3}{2},\frac{n}{2}+2}{\frac{4}{3},\frac{5}{3}},
\end{align}
where we have used $\dis\pochhammer{\frac{4}{3}}\pochhammer{\frac{5}{3}}=\frac{\pochhammer[2n+2]{n+1}}{2\cdot 27^n}$. 
So, we have $\dis c'_n=\frac{n+5}{2}$, for any $n\in\N$, and the recurrence relation coefficients given by \eqref{betas - hypergeometric weights}-\eqref{gammas - hypergeometric weights} are all constant and equal to the limits in \eqref{limits of the recurrence relation coefficients}, precisely we have: 
\begin{align*}
\beta_n\left(\frac{4}{3},\frac{5}{3};2,\frac{5}{2}\right)=\frac{4}{9},
\quad
\alpha_{n+1}\left(\frac{4}{3},\frac{5}{3};2,\frac{5}{2}\right)=\frac{16}{243}
\quad\text{and}\quad
\gamma_{n+1}\left(\frac{4}{3},\frac{5}{3};2,\frac{5}{2}\right)=\frac{64}{19683}, 
\end{align*}
for all $n\in\N$. 
Therefore, based on Theorem \ref{recurrence relation satisfied by the 2-OPS - 1st theorem}, the sequence $\dis\polyseq[P_n\left(x;\frac{4}{3},\frac{5}{3};2,\frac{5}{2}\right)]$ satisfies the third-order recurrence relation with constant coefficients
\begin{align}
\label{recurrence relation with constant coefficients}
P_{n+1}(x)=\left(x-\frac{4}{9}\right)P_n(x)-\frac{16}{243}\,P_{n-1}(x)-\frac{64}{19683}\,P_{n-2}(x).
\end{align} 
Finally, recall  \eqref{hypergeometric weight vector} and \eqref{hypergeometric weight definition} and use \cite[Eq.~15.4.9]{DLMF} to conclude that the polynomials in \eqref{explicit expression for the 2-OPS with constant recurrence coefficients}-\eqref{constant coeffs weights} are $2$-orthogonal with respect to the vector of weights
\begin{align}\label{constant coeffs weights}
\Wvec\left(x;\frac{4}{3},\frac{5}{3};2,\frac{5}{2}\right)=
\twovector{\frac{81\sqrt{3}}{16\pi}\,x^{\frac{1}{3}} \left(\left(1+\sqrt{1-x}\right)^{\frac{1}{3}}-\left(1-\sqrt{1-x}\right)^{\frac{1}{3}}\right)} 
{\frac{243\sqrt{3}}{160\pi}\,x^{\frac{1}{3}} \left(\left(1+\sqrt{1-x}\right)^{\frac{4}{3}}-\left(1-\sqrt{1-x}\right)^{\frac{4}{3}}\right)}.
\end{align}
Observe the similarities between the orthogonality weights above and the asymptotic zero distribution \eqref{asymptotic zero distribution for the 2-OPS wrt the hypergeometric weights}.

\subsection{Connection with other Hahn-classical 2-orthogonal polynomials}
\label{Hahn-classical 3-fold symmetric 2-OPS}

Particular choices on the parameters $a,b,c$ and $d$ of the 2-orthogonal polynomial sequence \eqref{explicit formula for the 2-orthogonal polynomials as a 3F2} appeared in  \cite{AnaWalter3FoldSym} as the components of a certain family of {\it threefold symmetric Hahn-classical} $2$-orthogonal polynomials on star-like sets. 
A polynomial sequence $\dis\polyseq[S_n(x)]$ is said to be threefold symmetric if
\begin{align}
S_n\left(\e^{\frac{2\pi }{3}i}\ x\right)=\e^{\frac{2n\pi }{3}i}S_n(x)
\text{ and }
S_n\left(\e^{\frac{4\pi }{3}i}\ x\right)=\e^{\frac{4n\pi }{3}i}S_n(x).
\ \text{ for all }\ n\in\mathbb{N}.  
\end{align}
This means there exist three polynomial sequences $\dis\polyseq[S_n^\kk(x)]$, with $k\in\{0,1,2\}$, which are called the \textit{cubic components} of $\dis\polyseq[S_n(x)]$, such that 
\begin{equation}
\label{cubic comp}
S_{3n+k}(x)=x^k S_n^\kk(x^3) \ \text{ for all }\ n\in\mathbb{N}.
\end{equation} 

As reported in \cite{DouakandMaroniClassiquesDeDimensionDeux} and studied in detail in \cite{AnaWalter3FoldSym}, there are four distinct families of Hahn-classical threefold symmetric $2$-orthogonal polynomials, up to a linear transformation of the variable. The four arising cases were therein denominated as A, B1, B2 and C. 
The polynomials in case A have no parameter dependence and their cubic components are particular cases of the $2$-orthogonal polynomials with respect to Macdonald functions investigated in \cite{BenCheikhDouak} and \cite{SemyonWalter}, while the polynomials in cases B1 and B2 depend on a parameter and their cubic components are particular cases of the $2$-orthogonal polynomials with respect to weights involving confluent hypergeometric functions of the second kind. These components were investigated in \cite{PaperTricomiWeights}.
At last, the cubic components of the polynomials in case C, depend on two parameters, are particular cases of the $2$-orthogonal polynomials under analysis here. 

Precisely, denoting the Hahn-classical threefold symmetric $2$-orthogonal polynomials analysed in \cite[\S 3.4]{AnaWalter3FoldSym} by $S_n(x;\mu,\rho)$ and their cubic components by $S_n^\kk(x;\mu,\rho)$, $k\in\{0,1,2\}$, and comparing the explicit expressions exhibited in \cite[\S 3.4.1]{AnaWalter3FoldSym} with \eqref{explicit formula for the 2-orthogonal polynomials as a 3F2}, we derive that, for each $\mu,\rho\in\R^+$ and $k\in\{0,1,2\}$,
\begin{align*}
S_n^\kk(x;\mu,\rho)=P_n\left(x;a_k,b_k;c_k,d_k\right),
\end{align*}
with
$ \left(a_k,b_k;c_k,d_k\right)$ equal to 
$ \left(\frac{1}{3},\frac{2}{3};\frac{\mu+2}{3},\frac{\rho}{3}+1\right)$,
$ \left(\frac{4}{3},\frac{2}{3};\frac{\rho}{3}+1,\frac{\mu+5}{3}\right)$ and
$ \left(\frac{4}{3},\frac{5}{3};\frac{\mu+5}{3},\frac{\rho}{3}+2\right)$,
for $k=0,1,2$, respectively.

Furthermore, there are confluent relations between the 2-orthogonal polynomials analysed here and the ones investigated in \cite{PaperTricomiWeights}.
These relations generalise the ones between case C and cases B1 and B2 in \cite{AnaWalter3FoldSym}, similarly to how the confluent relations shown in \cite[Section~3.5]{PaperTricomiWeights} generalise the ones between cases B1 and B2 and case A.

The 2-orthogonal polynomials investigated in \cite{PaperTricomiWeights} satisfy orthogonality conditions with respect to weight functions $\V(x;a,b;c)$ and $\V(x;a,b;c+1)$, supported in $\R^+$, with $a,b,c\in\R^+$ such that $c>\max\{a,b\}$ and
\begin{align}
\label{Tricomi weight definition}
\V(x;a,b;c)=\frac{\Gamma(c)}{\Gamma(a)\Gamma(b)}\,\e^{-x}x^{a-1}\,\KummerU{c-b}{a-b+1},
\end{align}
where $\dis\KummerU[x]{\alpha}{\beta}$ is the confluent hypergeometric function of the second kind, also known as the Tricomi function (see \cite[\S 13]{DLMF} for the definition and some properties of the confluent hypergeometric functions). These $2$-orthogonal polynomials have also appeared in \cite{DouakMaroni2020}.

As shown in \cite[Th.~3.1]{PaperTricomiWeights}, if we denote by $R_n^\parameter{\epsilon}(x;a,b;c)$, with $\epsilon\in\{0,1\}$, the 2-orthogonal polynomials with respect to $\dis\left[\V(x;a,b;c+\epsilon),\V(x;a,b;c+1-\epsilon)\right]$, then
\begin{align}
\label{explicit formula for the 2-orthogonal polynomials as a 2F2}
R_n^\parameter{\epsilon}(x)
=\frac{(-1)^n\pochhammer{a}\pochhammer{b}}{\pochhammer{c+\floor{\frac{n+\epsilon}{2}}}}
\,\twoFtwo{-n,c+\floor{\frac{n+\epsilon}{2}}}{a,b}.
\end{align}

The confluent relations are a straightforward consequence of the explicit expressions for the $2$-orthogonal polynomials via the confluent relation for the generalised hypergeometric series \eqref{confluent relation for the generalised hypergeometric series}.
Naturally, limiting relations connecting the corresponding weight functions are obtained in a similar manner.

So, applying the confluent relation for the generalised hypergeometric series \eqref{confluent relation for the generalised hypergeometric series} to the polynomials defined by \eqref{explicit formula for the 2-orthogonal polynomials as a 3F2}, we derive the confluent relations
\begin{align}
\label{confluent relations between the 2-OPS wrt to the hypergeometric wrights and the Tricomi weights}
\lim_{d\to\infty}P_n\left(\frac{x}{d};a,b;c,d\right)=R_n^\parameter{0}\left(x;a,b;c\right)
\quad\text{and}\quad
\lim_{c\to\infty}P_n\left(\frac{x}{c};a,b;c,d\right)=R_n^\parameter{1}\left(x;a,b;d-1\right).
\end{align}
We can also obtain similar confluent relations connecting the weight functions $\W(x;a,b;c,d)$, defined by \eqref{hypergeometric weight definition}, and $\V(x;a,b;c)$, defined as in \eqref{Tricomi weight definition}.
More precisely, we derive
\begin{align}
\label{confluent relations between the the hypergeometric wrights and the Tricomi weights}
\lim_{d\to\infty}\frac{1}{d}\,\W\left(\frac{x}{d};a,b;c,d\right)=\V\left(x;a,b;c\right)
\quad\text{and}\quad
\lim_{c\to\infty}\frac{1}{c}\,\W\left(\frac{x}{c};a,b;c,d\right)=\V\left(x;a,b;d\right),
\end{align} 
as a consequence of combining the linear transformation of variable (see \cite[Eq.~15.8.1]{DLMF})
\begin{align}
\twoFone[z]{\alpha,\gamma-\beta}{\gamma}=\twoFone[\frac{z}{z-1}]{\alpha,\beta}{\gamma}
\end{align}
and the limiting relation between the hypergeometric and Tricomi functions (see \cite[Eq.~6.8.1]{BatemanProjectVol1})
\begin{align}
\lim_{\gamma\to\infty}\twoFone[1-\frac{\gamma}{x}]{\alpha,\beta}{\gamma}=x^{\alpha}\KummerU{\alpha}{\alpha-\beta+1}.
\end{align}

\bigskip

\subsection*{Concluding remarks.}

The main contribution of this paper is the analysis of the multiple orthogonal polynomials on the step line with respect to the Nikishin system obtained in Subsection \ref{Nikishin system}.
The study of the multiple orthogonal polynomials with respect to the same system for indices out of the step line and, in particular, the study of the (standard) orthogonal polynomials with respect to the weight function $\W(x;a,b;c,d)$ remains an open (and challenging) problem.
The same holds in general when the weight function is a solution to a second (or higher) order differential equation. 
In spite of this, the knowledge of the multiple orthogonal polynomials whose indexes lie on the step line is a largely sufficient tool for its applicability to a number of related fields in mathematics.
An example of this applicability is the newly found connection between multiple orthogonal polynomials and branched continued fractions   which will be object of further research. 

\bigskip

\textbf{Acknowledgements:}
We are grateful to Alex Dyachenko for kindly sharing with us a relevant formula from his joint work with Dmitrii Karp \cite{DyachenkoKarp}, which is in preparation for publication and is not publicly available yet.
We also thank Erik Koelink for the suggestion to look at the particular case with constant recurrence coefficients, and Alan Sokal for illuminating discussions on branched continued fractions and their connection with multiple orthogonal polynomials.

\bibliographystyle{plain}
\bibliography{MOPS}

\begin{thebibliography}{10}

\bibitem{AptekarevKaliaguine1998}
A.I. Aptekarev and V.A. Kalyagin.
\newblock Complex rational approximation and difference operators.
\newblock {\em {Rend. Circ. Matem. Palermo, Ser. II, Suppl}}, 52:3--21, 1998.

\bibitem{AptKalLagoRochaLimitBehaviour}
A.I. Aptekarev, V.A. Kalyagin, G.~L\'opez~Lagomasino, and I.A. Rocha.
\newblock On the limit behavior of recurrence coefficients for multiple
  orthogonal polynomials.
\newblock {\em J. Approx. Theory}, 13:779--811, 2011.

\bibitem{BenCheikhDouak}
Y.~Ben~Cheikh and K.~Douak.
\newblock On two-orthogonal polynomials related to the {Bateman}'s
  ${J}_n^{u,v}$-function.
\newblock {\em Methods Appl. Anal.}, 7(4):641--662, 2000.

\bibitem{PrudnikovEtAlVol3}
Yu.~A. Brychkov, O.~I. Marichev, and A.~P. Prudnikov.
\newblock {\em {Integrals and Series Volume 3 - More Special Functions}}.
\newblock 1990.

\bibitem{WalterCoussementX2AsymptoticZeroDistribution}
E.~Coussement, J.~Coussement, and W.~Van~Assche.
\newblock Asymptotic zero distribution for a class of multiple orthogonal
  polynomials.
\newblock {\em Trans. Amer. Math. Soc.}, 360(10):5571--5588, 2008.

\bibitem{WalterCoussementSomeClassicalMOPs}
E.~Coussement and W.~Van~Assche.
\newblock Some classical multiple orthogonal polynomials.
\newblock {\em J. Comput. Appl. Math.}, 127:317--347, 2001.

\bibitem{ContinuedFractionsForSpecialFunctions}
A.~Cuyt, W.B. Jones, V.B. Petersen, B.~Verdonk, and H.~Waadeland.
\newblock {\em Handbook of {Continued} {Fractions} for {Special} {Functions}}.
\newblock Springer, New York, 2008.

\bibitem{DLMF}
{\it NIST Digital Library of Mathematical Functions}.
\newblock http://dlmf.nist.gov/, Release 1.0.17 of 2017-12-22.
\newblock F.W.J. Olver, A.B.{Olde Daalhuis}, D.W. Lozier, B.I. Schneider, R.F.
  Boisvert, C.W. Clark, B.R. Miller and B.V. Saunders, eds.

\bibitem{DouakandMaroniClassiquesDeDimensionDeux}
K.~Douak and P.~Maroni.
\newblock Les polyn\^omes orthogonaux classiques de dimension deux.
\newblock {\em Analysis}, 12:71--107, 1992.

\bibitem{DouakMaroni2020}
Khalfa Douak and Pascal Maroni.
\newblock On a new class of 2-orthogonal polynomials, i: the recurrence
  relations and some properties.
\newblock {\em Integral Transforms and Special Functions}, 0(0):1--20, 2020.

\bibitem{DyachenkoKarp}
A.~Dyachenko and D.~Karp.
\newblock {Ratios of the Gauss hypergeometric functions with parameters shifted
  by integers}.
\newblock {\em In preparation for publication}.

\bibitem{BatemanProjectVol1}
A.~Erdelyi, W.~Magnus, F.~Oberhettinger, and F.~G. Tricomi.
\newblock {\em {Higher Transcendental Functions}}, volume~1.
\newblock McGraw-Hill, New York, 1953.

\bibitem{NikishinSystemsArePerfect}
U.~Fidalgo~Prieto and G.~L\'opez~Lagomasino.
\newblock Nikishin systems are perfect.
\newblock {\em Constr. Approx.}, 34:297--356, 2011.

\bibitem{NikishinSystemsArePerfectCaseOfUnboundedAndTouchingSupports}
U.~Fidalgo~Prieto and G.~L\'opez~Lagomasino.
\newblock Nikishin systems are perfect. {Case} of unbounded and touching
  supports.
\newblock {\em J. Approx. Theory}, 13:779--811, 2011.

\bibitem{GantmacherKreinOscillationMatrices}
F.R. Gantmacher and M.G. Krein.
\newblock {\em {Oscillation Matrices and Kernels and Small Vibrations of
  Mechanical Systems (Revised English Edition)}}.
\newblock AMS Chelsea Publishing, Providence, R.I., 2002.

\bibitem{GradshteynRyzhik}
I.S. Gradshteyn and I.~M. Ryzhik.
\newblock {\em {Table of Integrals, Series, and Products}}.
\newblock Academic Press, New York, 7 edition, 2007.

\bibitem{IsmailBook}
M.E.H. Ismail.
\newblock {\em Classical and {Quantum} {Orthogonal} {Polynomials} in {One}
  {Variable}}.
\newblock Cambridge {University} {Press}, 2005.

\bibitem{KuijlaarsStiv14}
A.B.J. Kuijlaars and D.~Stivigny.
\newblock Singular values of products of random matrices and polynomials
  ensembles.
\newblock {\em Random Matrices: Theory Appl.}, 3(3, 1450011):22 pp., 2014.

\bibitem{KuijlaarsZhang14}
A.B.J. Kuijlaars and L.~Zhang.
\newblock Singular values of products of {Ginibre} random matrices, multiple
  orthogonal polynomials and hard edge scaling limits.
\newblock {\em Comm. Math. Phys.}, 2(332):759--781, 2014.

\bibitem{Kustner}
R.~K\"ustner.
\newblock Mapping properties of hypergeometric functions and convolutions of
  starlike or convex functions of order $\alpha$.
\newblock {\em Comput. Methods Funct.Theor}, 2(2):597--610, 2002.

\bibitem{LamiriOuni}
I.~Lamiri and A.~Ouni.
\newblock {$d$-Orthogonality of Humbert and Jacobi type polynomials}.
\newblock {\em J. Math. Anal. Appl.}, 341:24--51, 2008.

\bibitem{PaperTricomiWeights}
H.~Lima and A.~Loureiro.
\newblock Multiple orthogonal polynomials associated with confluent
  hypergeometric functions.
\newblock arXiv:2001.06820, 2020.

\bibitem{GuillermoSurvey}
G.~L\'opez~Lagomasino.
\newblock An introduction to multiple orthogonal polynomials and
  {Hermite}-{Pad\'e} approximation.
\newblock arXiv:1910.08548, 2019.

\bibitem{AnaWalter3FoldSym}
A.~Loureiro and W.~Van~Assche.
\newblock Threefold symmetric {Hahn}-classical multiple orthogonal polynomials.
\newblock {\em Anal. Appl.}, 18(2):271--332, 2020.

\bibitem{MaroniOrthogonalite}
P.~Maroni.
\newblock L'orthogonalit\'e et les r\'ecurrences de polyn\^omes d'ordre
  sup\'erieur \`a deux.
\newblock {\em Ann. Fac. Sci. Toulouse}, 10(1):105--139, 1989.

\bibitem{Minton}
B.M. Minton.
\newblock Generalized hypergeometric functions at unit argument.
\newblock {\em J. Math. Phys.}, 12:1375--1376, 1970.

\bibitem{NikishinSystems}
E.M Nikishin.
\newblock On simultaneous {Pad\'e} approximants.
\newblock {\em Mat. Sb. (N.S.)}, 41:409--425, 1982.

\bibitem{NikishinSorokinBook}
E.M. Nikishin and V.N. Sorokin.
\newblock {\em {Rational Approximations and Orthogonality}}, volume~92 of {\em
  {Translations of Mathematical Monographs}}.
\newblock {Amer. Math. Soc., Providence, R.I.}, 1991.

\bibitem{AlanSokalEtAlBranchedContinuedFractions}
M.~P\'etr\'eolle, A.~Sokal, and B.~Zhu.
\newblock {Lattice paths and branched continued fractions: {An} infinite
  sequence of generalizations of the {Stieltjes-Rogers} and {Thron-Rogers}
  polynomials, with coefficientwise {Hankel}-total positivity}.
\newblock arXiv:1807.03271, 2018.

\bibitem{Pineiro}
L.R. Pi\~neiro.
\newblock On simultaneous approximations for a collection of {Markov}
  functions.
\newblock {\em Vestnik Mosk. Univ., Ser. I}, 2:67--70, 1987.

\bibitem{SaibNewPerspectivesOnD-OPS}
A.~Saib.
\newblock Some new perspectives on d-orthogonal polynomials.
\newblock arXiv:1605.00049v5, 2018.

\bibitem{Stieltjesmemoir}
T.-J. Stieltjes.
\newblock Recherches sur les fractions continues.
\newblock {\em Annales de la Facult\'e des Sciences de Toulouse pour les
  Sciences Math\'ematiques et les Sciences Physiques}, 8(4):1--122, 1894.

\bibitem{WalterNonsymmetric}
W.~Van~Assche.
\newblock Nonsymmetric linear difference equations for multiple orthogonal
  polynomials.
\newblock {\em SIDE II in: CRM Proc. Lecture Notes}, 25:415--429, 2000.

\bibitem{WalterNearestNeighborRecurrenceRelations}
W.~Van~Assche.
\newblock Nearest neighbor recurrence relations for multiple orthogonal
  polynomials.
\newblock {\em J. Approx. Theory}, 163(10):1427--1448, 2011.

\bibitem{WalterMehlerHeineAsymptoticsForMOPs}
W.~Van~Assche.
\newblock {Mehler-Heine asymptotics for multiple orthogonal polynomials}.
\newblock {\em Proc. Amer. Math. Soc.}, 145:303--314, 2017.

\bibitem{SemyonWalter}
W.~Van~Assche and S.~Yakubovich.
\newblock Multiple orthogonal polynomials associated with {Macdonald}
  functions.
\newblock {\em Integral Transforms Spec. Funct.}, 9(3):229--244, 2000.

\bibitem{WallContinuedFractions}
H.S. Wall.
\newblock {\em {Analytic Theory of Continued Fractions}}.
\newblock Chelsea, Bronx, NY, 1973.

\end{thebibliography}

\end{document}